\documentclass[12pt]{elsarticle}
\usepackage{amsfonts}
\usepackage{amsmath,cancel}
\usepackage{amssymb}
\usepackage{xcolor}
\usepackage{graphicx}
\usepackage{mathrsfs}
\usepackage[normalem]{ulem}
\newtheorem{theorem}{Theorem}[section]

\newtheorem{teo}[theorem]{Theorem}
\newtheorem{obs}[theorem]{Observation}
\newtheorem{lema}[theorem]{Lemma}

\newtheorem{ej}[theorem]{Example}
\newtheorem{conjecture}[theorem]{Conjecture}
\newtheorem{cor}[theorem]{Corollary}

\newtheorem{definition}[theorem]{Definition}

\newtheorem{lemma}[theorem]{Lemma}

\newtheorem{prop}[theorem]{Proposition}

\newenvironment{proof}[1][\noindent Proof]{\textbf{#1.} }{\ \rule{0.5em}{0.5em}}
\newcommand{\N}{\mathbb{N}}
\newcommand{\R}{\mathbb{R}}
\newcommand{\M}{\mathcal{M}}

\DeclareMathOperator{\Irr}{Irr}

\title{Some families of digraphs determined by the complementarity spectrum}

\author{Diego Bravo}
\ead{dbravo27@gmail.com}
\address{PEDECIBA, Uruguay.}

\author{Florencia Cubr\'{\i}a\footnote{Corresponding author: fcubria@fing.edu.uy, Herrera y Reissig 565, Montevideo, Uruguay.}}
\ead{fcubria@fing.edu.uy}
\author{Marcelo Fiori}
\ead{mfiori@fing.edu.uy}
\author{Gustavo Rama}
\ead{grama@fing.edu.uy}
\address{Instituto de Matemática y Estadística ``Rafael Laguardia'', Facultad de Ingenier\'{i}a,\\ Universidad de la Republica, Uruguay.}

\begin{document}

\begin{abstract}
We examine the capacity of the complementarity spectrum to distinguish non-isomorphic digraphs. We focus on the seven families with exactly three complementarity eigenvalues \cite{flor2}. Our findings reveal that in some, but not all families, any two non-isomorphic members have different complementarity spectrum.\\
Complementarity eigenvalues outperform traditional eigenvalues in the task of identifying graphs. Indeed, the question of whether graphs are uniquely determined by their complementarity spectrum remains unresolved, highlighting the significance of this tool in graph theory.
Moreover, since the complementarity spectrum of a digraph was characterized in \cite{flor} as the set of spectral radii of the induced strongly connected subdigraphs, the results of this study provides useful structural information for important families of digraphs.

\end{abstract}


\maketitle

\section{Introduction}
The discriminative power of the spectral properties of adjacency matrices, along with other matrices related to graphs and digraphs, has been an active field of research for decades. Recently, the complementarity eigenvalues have shown a considerable advantage in the unique identification of graphs, contrasting with traditional eigenvalues. This advantage is significant, as Collatz and Sinogowitz demonstrated the limitations of traditional eigenvalues in uniquely determining graphs through the adjacency matrix's eigenvalues \cite{Collatz1957}. This limitation also extends 
to directed graphs (digraphs), where Harary provided the first instance of cospectral non-isomorphic digraphs \cite{Harary1971}. Yet, the question of whether undirected graphs can be uniquely identified by their complementarity spectrum remains unresolved, underscoring its critical role in graph theory.\\

The concept of the complementarity spectrum for digraphs was introduced in \cite{flor}, where also it was observed the phenomenon of non-isomorphic digraphs sharing identical spectra. Despite this limitation, further investigations into this novel concept, particularly the exploration of which digraph families can be distinguished based on their complementarity spectrum, remain imperative. This mirrors an earlier study for undirected graphs \cite{Pinheiro2020}.\\

In the context of digraphs, the number of complementarity eigenvalues is indicative of the underlying structure: a single eigenvalue suggests an acyclic graph, whereas two eigenvalues point to a structure comprising cycles or isolated vertices \cite{flor}. Strongly connected digraphs characterized by exactly three complementarity eigenvalues have been classified into seven distinct families \cite{flor2}, illustrating the close relationship between a digraph's structure and its complementarity spectrum.\\

This research investigates the capability to distinguish among digraphs across seven specified families based on their complementarity spectra. It reveals that while the complementarity spectra can uniquely identify certain families, others remain indistinguishable.\\

Among the seven families studied, it particularly advances and completes the preliminary findings of \cite{Lin2012} for $\theta$-digraphs, a bicyclic family among them. It establishes that these digraphs, along with others, can be distinguished by their spectral radii.\\

Moreover, this work enhances the understanding of digraph structures within these families by analyzing their characteristic polynomials and spectral radii. This analysis is underpinned by characterizing the complementarity spectrum through the spectral radii of induced strongly connected sub-digraphs \cite{flor}.\\

The broader implications of complementarity or Pareto spectra over symmetric matrices, extend to linear programming, cone-constrained dynamical systems and mathematical modeling in general \cite{Baillon2020}. Expanding the study of complementarity spectrum to non-symmetric matrices opens new avenues for applying complementarity eigenvalues.\\ 

The cardinality of the complementarity spectrum, not tied to the (di)graph's order, introduces unique challenges. Initial research has thoroughly investigated digraphs with up to three complementarity eigenvalues \cite{flor2}, identifying seven families with specific characteristics. This study leverages these foundational classifications to determine the isomorphism of digraphs within these families based on their shared complementarity spectrum, employing both classical and algebraic analysis of characteristic polynomials and spectral radii.\\

Following this introduction, the paper is systematically organized into sections detailing preliminary concepts, analyses of families derived from cycle coalescence and $\theta$-digraphs, and concludes with discussions on the uniqueness of the complementarity spectrum among these families, along with conjectures for further investigation.\\

\section{Preliminaries and notation}
\label{sec:preliminares}

Let $D=(V,E)$ be a finite simple digraph  with vertices labeled as $1, \hdots, n$. The adjacency matrix of $D$ is defined as $A(D)=(a_{ij})$ where
\[a_{ij}=
\begin{cases}
1 \quad \text{if }(i,j)\in E,\\
0 \quad \text{otherwise.}
\end{cases}
\]

The \textit{spectrum} of $D$ is the multi-set of roots of the characteristic polynomial of $A(D)$, counted with their multiplicities, denoted by $Sp(D)$. If $D$ is a digraph with strongly connected components $D_1, \hdots,D_k$, then $Sp(D)=\bigsqcup_{i=1}^k Sp(D_i)$ where $\bigsqcup$ denotes the union of multi-sets.\\

Throughout this paper, $\rho(\cdot)$ denotes the spectral radius, i.e., the largest module of the eigenvalues of a matrix. For non negative irreducible matrices, the spectral radius coincides with the largest real eigenvalue, by virtue of the Perron-Frobenius theorem. Additionally, in this case, the spectral radius is simple and may be associated with an eigenvector $x>\textbf{o}$, where $\textbf{o}$ denotes the null vector in $\R^n$ and $\geq$ {(or $>$)} means that the inequality holds for every coordinate.\\

This real positive value $\rho(A(D))$ is called the \emph{spectral radius of the digraph $D$} and is denoted by $\rho(D)$. As we will see, the spectral radius of the digraph plays a fundamental role in the results obtained in this paper.\\

A digraph $H=(V',E')$ is a sub-digraph of $D$ (denoted $H\leq D$) if $V'\subseteq V$ and $E'\subseteq E$. We say that $H$ is an induced sub-digraph if $E'=E \cap (V'\times V')$ and a proper sub-digraph if $E'\neq E$.\\

The following Lemma is a consequence of the Perron-Frobenius theorem, but we state it here for easy reference.

\begin{lemma}\label{lem:sub} Let $H$ be a proper sub-digraph of a strongly connected digraph $D$. Then $\rho(H) < \rho(D)$.
\end{lemma}

We use the term \textit{cycle} to refer to a directed cycle in a digraph. As noted in \cite{flor}, the spectral radius of the cycle digraph $\vec{C}_n$, is $\rho(\vec{C}_n)=1.$\\

The Eigenvalue Complementarity Problem (EiCP) introduced in \cite{Seeger99}, has found many applications in different fields of science, engineering and economics \cite{Adly2015,Facchinei2007,Pinto2008,Pinto2004}.\\

Given a matrix $A \in \M_n(\mathbb{R})$, the set of \textit{complementarity eigenvalues} is defined as those $\lambda \in \R$ such that there exists a vector $x \in \R^n$, not null and non negative, verifying
 $Ax \geq \lambda x$, and
\begin{equation*}
\label{complement}
\langle  x, Ax-\lambda x \rangle =0.
\end{equation*}

Writing $w=Ax-\lambda x\geq \textbf{o}$, the previous condition results in\[x^tw=0 \]
which means
\[x_i=0 \quad \text{or} \quad w_i=0 \text{ for all } i=1 \hdots, n.\]

This last condition is called \textit{complementarity condition}.\\

The set of all complementarity eigenvalues of a matrix $A$ is called the \textit{complementarity spectrum} of $A$, and it is denoted $\Pi(A)$. Unlike the regular spectrum of a matrix, the complementarity spectrum is a set (not a multiset), and the number of complementarity eigenvalues is not determined by the size of the matrix.\\

It is known that if $\lambda$ is a complementarity eigenvalue of $A$, then it is a complementarity eigenvalue of $PAP^t$ as well, for every permutation matrix $P$ \cite{Pinto2008}.\\

This fact allows to define the complementarity spectrum of a digraph, since the complementarity spectrum is invariant in the family of adjacency matrices associated to the digraph.\\

The following theorem \cite{flor} extends an existing result for graphs and leads to a simple useful characterization of the complementarity eigenvalues in terms of the structural properties of a digraph.

\begin{teo}
\label{compl_spect_induced_subdigraphs} Let $D$ be a digraph and $\Pi(D)$ its complementarity spectrum. Then
\[\Pi(D)=\{\rho(H): {H \text{ induced strongly connected subdigraph of }D}\}.\]
\end{teo}

Particularly, the complementarity spectrum of a digraph is then a set of non negative real numbers containing $0$.\\

The following result \cite{flor} allows to focus on strongly connected digraphs.
\begin{prop}\label{prop:strong}
Let $D$ be a digraph and $D_1, \hdots, D_k$ the digraph generated by the strongly connected components of $D$. Then,
\[
\Pi(D)=\cup_{i=1}^k \Pi(D_i).
\]
\end{prop}

Let $\#\Pi(D)$ denote the cardinality of $\Pi(D)$. The complementarity spectrum encodes structural properties of the digraph as it can been seen in the next result \cite{flor}.  

\begin{theorem}\label{tm:char} Let $D$ be a digraph and $\Pi(D)$ its complementarity spectrum. The three conditions listed under (1) are mutually equivalent, as are the three conditions under (2).
\begin{enumerate}
\item
\begin{enumerate}
    \item[(a)] $\Pi(D)=\{0\}$,
    \item[(b)] $\#\Pi(D)=1$,
    \item[(c)] $D$ is acyclic.
    \end{enumerate}
\item
    \begin{enumerate}
    \item[(a)] $\Pi(D)=\{0,1\}$,
    \item[(b)] $\#\Pi(D)=2$,
    \item[(c)] $D$ is not acyclic and its strongly connected components are either cycles or isolated vertices.
    \end{enumerate}
\end{enumerate}
\end{theorem}

The findings of Theorem~\ref{tm:char}, as detailed in \cite{flor}, conclusively identify digraphs based on their having either one or two complementarity eigenvalues.\\

Drawing from Theorem \ref{tm:char}, it is evident that any digraph containing a cycle as a sub-digraph invariably has 0 and 1 as its complementarity eigenvalues.\\

The set $\mathcal{SCD}_t$ denotes the collection of all strongly connected digraphs that have exactly $t$
complementarity eigenvalues. Notably, $\mathcal{SCD}_1$ comprises solely the digraph consisting of a single isolated vertex, while $\mathcal{SCD}_2$ encompasses all cycles, irrespective of their length.

\begin{theorem}
\label{caracterizacion}
{$\mathcal{SCD}_3$ is the set of strongly connected digraphs whose only proper induced strongly connected sub-digraphs are both isolated vertices and cycles.}
\end{theorem}

In this paper, our focus is limited to digraphs that have a maximum of three complementarity eigenvalues. While Theorem~\ref{tm:char} delineates a complete account of digraphs with one or two elements in their complementarity spectrum, Theorem~\ref{caracterizacion} only provides a structural characterization for these digraphs.\\

To delve deeper into the classification of digraphs determined by their complementarity spectrum, we employ a categorization into seven families within \(\mathcal{SCD}_3\), as introduced in \cite{flor2}.\\

This classification begins with two foundational digraph families: the \(\infty\)-digraph and the \(\theta\)-digraph. From these, five additional digraph types are derived.\\

In the figures throughout this manuscript, digraphs are represented using the following convention: a single arrow between two vertices indicates one arc joining them, while a double arrow signifies that there may be other vertices in the path joining them. For clarity, we will label the vertices in \(\vec{C}_r\) as \(1, 2, \hdots, r\), and those in \(\vec{C}_s\) as \(1', 2', \hdots, s'\), with \(1\) identified with \(1'\).\\

\textbf{\(\infty\), Type 1 and Type 2 digraphs}\\

The \(\infty\)-digraph, denoted as \(\infty(r,s)\), is defined by the coalescence of two cycles, \(\vec{C}_r\) and \(\vec{C}_s\).\\


Figure \ref{fig:infinito} shows one example of a digraph in this family and a schematic representation.

\begin{figure}[h!]
\centering
\includegraphics[scale=0.18]{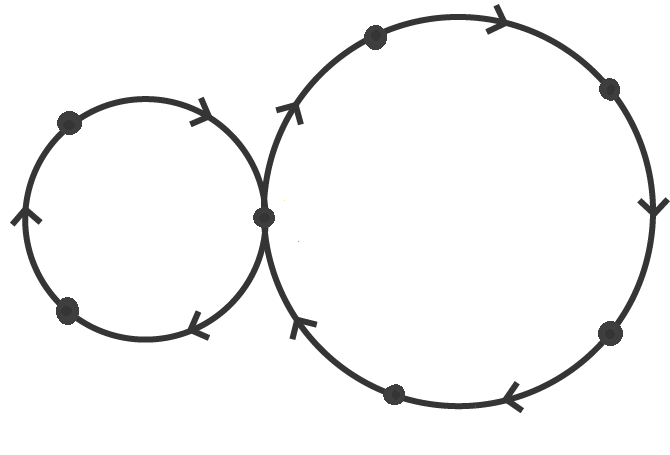}
\hspace{0.5cm}
\includegraphics[scale=0.18]{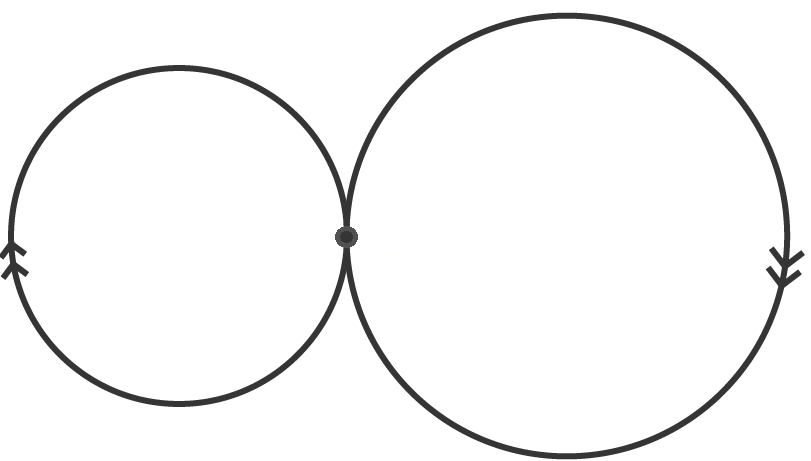}
\caption{Digraph $\infty(3,5)$ and a schematic representation of the same digraph.}
\label{fig:infinito}
\end{figure}

It is easy to see that the characteristic polynomial is $$p_{\infty}(x)=x^n-x^{s-1}-x^{r-1}.$$

The addition of one or two specific arcs to the \(\infty\)-digraph results in two other families, denoted as Type 1 and Type 2 digraphs, respectively. A Type 1 digraph is created by adding an arc that connects vertex \(r\) to vertex \(2'\). By further adding an arc from vertex \(s'\) to \(2\), a Type 2 digraph is obtained (see Figure \ref{fig_tipo1and2}).




\begin{figure}[h!]
\centering
\includegraphics[scale=0.15]{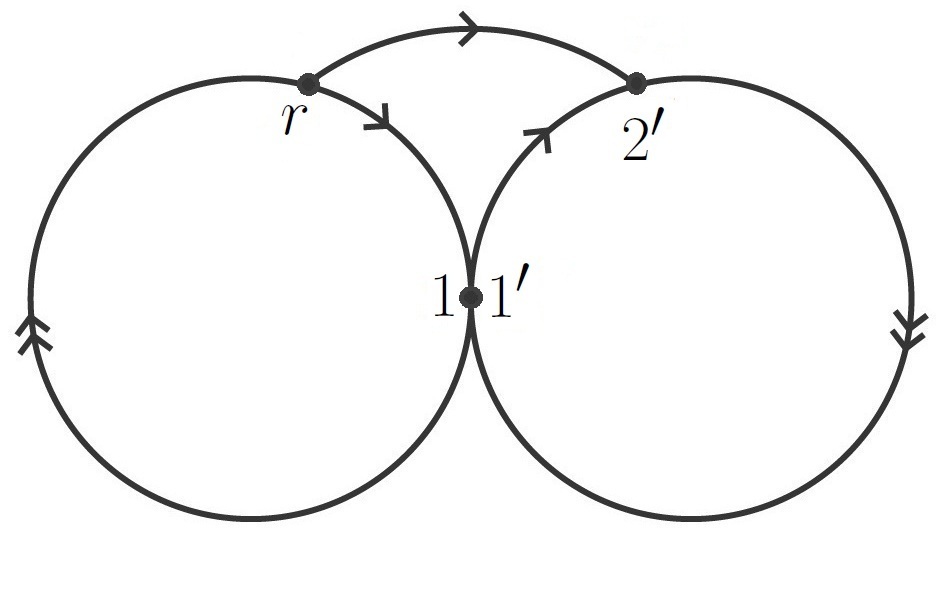}
\hspace{.5cm}
\includegraphics[scale=0.15]{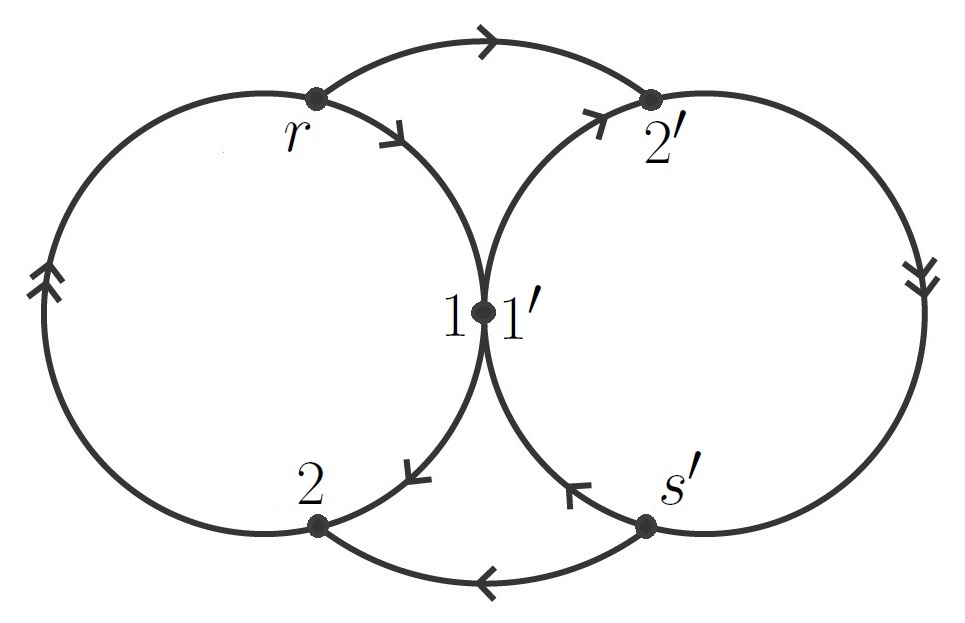}
\caption{Type 1 digraph (left) and Type 2 digraph (right)}.
\label{fig_tipo1and2}
\end{figure}




Let us denote by $D_1$ the Type 1 digraphs. Then, $D_1(r,s)=\infty(r,s)\cup\{e\}$, where $e=(r,2')$, as shown in Figure \ref{fig_tipo1and2} (left).\\

Similarly, let us denote by $D_2$ the Type 2 digraphs. Then, $D_2(r,s)=\infty(r,s)\cup\{e,\tilde{e}\}$, where $e=(r,2')$ and $\tilde{e}=(s',2)$, as shown in Figure \ref{fig_tipo1and2} (right). \\

It is readily apparent that, within these digraphs, the only strongly connected induced proper non-trivial sub-digraphs are cycles. Consequently, by virtue of Theorem~\ref{compl_spect_induced_subdigraphs}, the complementarity spectrum comprises three distinct elements: \(0\), \(1\), and the spectral radius of the digraph.
\\

By virtue of Sachs' Theorem we can see that the characteristic polynomials are 

\[p_{D_1}(x)=x^n-x^{n-r}-x^{n-s}-1=x^n-x^{s-1}-x^{r-1}-1,\]
\[p_{D_2}(x)=x^n-x^{n-r}-x^{n-s}-x-2=x^n-x^{s-1}-x^{r-1}-x-2.\]

Observe that among these families, the cycles in the \(\infty\) and Type 2 digraphs are interchangeable, indicating that \(\infty(r,s)\) is isomorphic to \(\infty(s,r)\), and a similar relationship holds for Type 2 digraphs. This property does not extend to Type 1 digraphs, necessitating a distinction into two sub-families as follows.\\

Given integers \(2 \leq r \leq s\), a Type $1a$ digraph is defined as \(D_{1a}(r,s) = \infty(r,s) \cup \{(r,2')\}\), and a Type $1b$ digraph as \(D_{1b}(r,s) = \infty(r,s) \cup \{(s',2)\}\).\\

In essence, Type $1a$ digraphs feature an added arc that connects the smaller cycle (in terms of vertex count) to the larger one, whereas in Type $1b$ digraphs, the added arc connects the larger cycle to the smaller one.\\

The digraph families introduced above each possess three complementarity eigenvalues: zero, one, and the spectral radius of the digraph itself. Applying the Perron-Frobenius theorem directly leads to a noteworthy observation.\\

\begin{obs}\label{rem:inftyFamily}
For the digraphs $\infty, D_{1}, D_{2}$ comprising \(n\) vertices, Lemma~\ref{lem:sub} implies that \(\rho(\infty) < \rho(D_{1}) < \rho(D_{2})\). This inequality sequence underscores the fact that these digraphs are distinct and not isomorphic to one another.
\end{obs}


It is a straightforward observation that in the \(\infty\), Type 1, and Type 2 digraph families, the only strongly connected induced sub-digraphs are cycles, besides the digraphs themselves and isolated vertices. Thus,

\[\Pi(\infty)=\{0,1,\rho(\infty)\}\]
\[\Pi(D_1)=\{0,1,\rho(D_1)\}\]
 \[\Pi(D_2)=\{0,1,\rho(D_2)\}\]\\



\noindent \textbf{$\theta$-digraph}\\



A \(\theta\)-digraph, as described in \cite{Lin2012,flor}, consists of three directed paths \(\vec{P}_{a+2}\), \(\vec{P}_{b+2}\), and \(\vec{P}_{c+2}\). These paths are arranged such that the initial vertex of \(\vec{P}_{a+2}\) and \(\vec{P}_{b+2}\) is the terminal vertex of \(\vec{P}_{c+2}\), and conversely, the initial vertex of \(\vec{P}_{c+2}\) is the terminal vertex of \(\vec{P}_{a+2}\) and \(\vec{P}_{b+2}\). This configuration is depicted in Figure \ref{fig:prohibido} and denoted by \(\theta(a, b, c)\), or simply \(\theta\).\\

\begin{figure}[h!]
\centering
\includegraphics[scale=0.15]{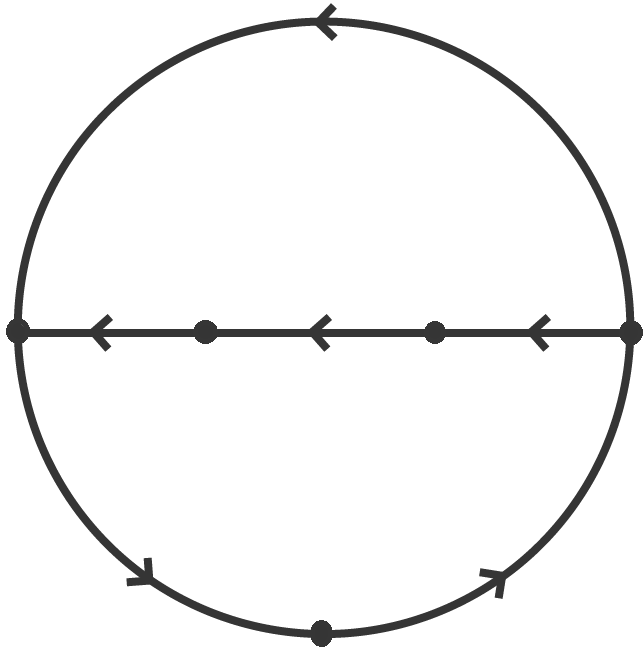}
\hspace{1cm}
\includegraphics[scale=0.15]{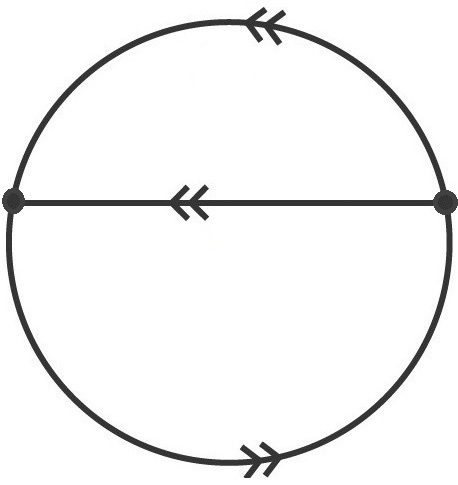}
\caption{Digraph $\theta(0,2,1)$ and a schematic representation of the same digraph.}
\label{fig:prohibido}
\end{figure}

Since ${\theta}(a,b,c)$ and ${\theta}(b,a,c)$ are isomorphic and we are not considering digraphs with multiple arcs, we can assume $a \leq b$ and $b > 0$,  without loss of generality.

Once again, the characteristic polynomial can be easily computed using Sachs' Theorem. \[p_{\theta}(x)=x^n-x^{n-r}-x^{n-s}= x^n-x^{b}-x^{a}.\]

The digraph $\theta(a,b,c)$ has $n=a+b+c+2$ vertices an the only strongly connected induced sub-digraphs are the cycles $\vec{C_r}$ and $\vec{C_s}$ (where $r=a+c+2$ and $s=b+c+2$), in addition to the digraph $\theta$ itself, and isolated vertices. Therefore,
\[\Pi({\theta})=\{0,1,\rho({\theta})\}.\]

Note that an equivalent construction of the $\theta$-digraph can be made by taking a cycle, and adding a simple path joining two different vertices.\\

We now describe three types of digraphs in $\mathcal{SCD}_3$
without $\infty$-subdigraphs.\\


\noindent \textbf{Type 3 digraph}\\

Let us denote by $D_3$ the Type 3 digraphs $D_3=\vec{C_n}\cup\{e,\tilde{e}\}$ where $e=(1,i)$ and $\tilde{e}=(i-1,j)$ with $2<i<j\leq n$, as shown in the Figure \ref{fig:tipo345} (left). We have that
\[p_{D_3}(x)=x^n-x^{i-2}-x^{j-i}-1.\]

The only strongly connected induced subdigraphs {of $D_3$} are the cycles $\vec{C}_{n-(i-2)}$, $\vec{C}_{n-(j-i)}$ besides the digraph $D_3$ and isolated vertices. Therefore,
\[\Pi(D_3)=\{0,1,\rho(D_3)\}.\]

\noindent \textbf{Type 4 digraph}\\

Let {$k$} be a positive integer larger than $1$. Let us denote by $D_4$ the Type 4 digraph $D_4=\vec{C_n}\cup\{e_1,e_2, \hdots, e_k\}$ where $e_i=(x_i,y_i)$ with $1<y_i<x_i<y_{i+1}<x_{i+1}\leq n$ for all $i=1, \hdots, k-1$. {Observe that the condition $y_i<x_i$ prevents the appearance of an $\infty$-sub-digraph.} Figure \ref{fig:tipo345} (center) shows such a digraph with $k=3$ added arcs, generating the three cycles $\vec{C}_{r_1}, \vec{C}_{r_2}$ and $\vec{C}_{r_3}$.\\

The only strongly connected induced sub-digraphs of $D_4$ are the cycles $\vec{C}_{r_1},\hdots,\vec{C}_{r_k}$  (where $C_{r_i}$ is the digraph induced {by the vertices in-between} $y_i$ and $x_i$ for all $i=1, \hdots, k$) besides the digraph $D_4$ and isolated vertices. Then, 
\[\Pi(D_4)=\{0,1,\rho(D_4)\}.\]


\noindent \textbf{Type 5 digraph}\\

Let us denote by $D_5$ the Type 5 digraphs $D_5=\vec{C_n}\cup\{e,e',e''\}$ where $e=(1,i)$, $e'=(i-1,j)$ and $e''=(j-1,2)$ with $3<i<i+1<j\leq n$. In this case, condition $3<i$ as much as condition $i+1<j$ prevents the appearance of an $\infty$-sub-digraph. Figure \ref{fig:tipo345} (right) shows an example of a digraph in this family.

The only strongly connected induced sub-digraphs {of $D_5$} are the cycles $\vec{C}_{n-(i-2)}$, $\vec{C}_{n-(j-i)}$ and $\vec{C}_{j-2}$ besides the digraph $D_5$ and isolated vertices. Then, we have
\[\Pi(D_5)=\{0,1,\rho(D_5)\}.\]

\begin{figure}[h!]
\centering
\includegraphics[scale=0.15]{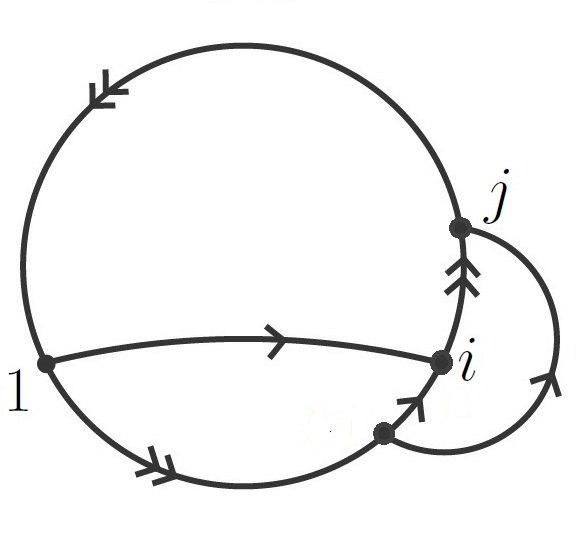}
\hspace{.5cm}
\includegraphics[scale=0.12]{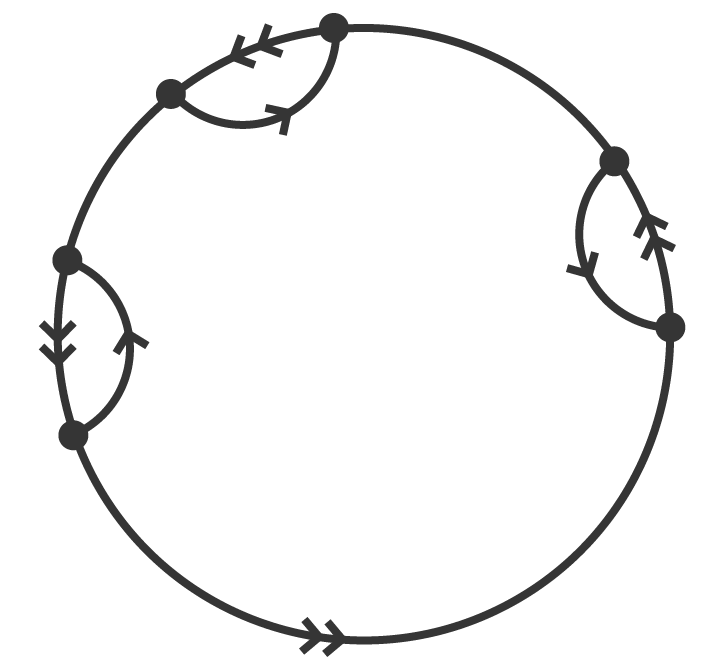}
\hspace{.5cm}
\includegraphics[scale=0.12]{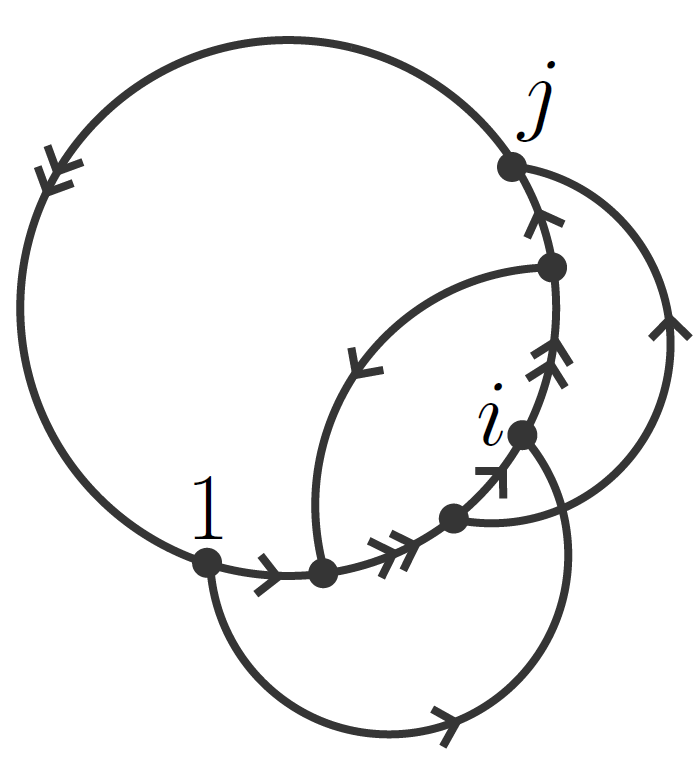}
\caption{Type 3 digraph (left), Type 4 digraph (center) and Type 5 digraph (right). {Observe that one $\theta$-subdigraph can be identified in these three examples by taking the larger (round) cycle in each case, and any other arc.}}
\label{fig:tipo345}
\end{figure}

It is easy to see that the seven types of digraphs presented above are non-isomorphic to each other.\\

\noindent \textbf{All digraphs with complementarity spectrum of size 3}\\

Let \(F_{\infty}\) denote the set of digraphs that are either \(\infty\)-digraphs or Type \(i\) digraphs for \(i \in \{1, 2\}\), and let \(F_{\theta}\) represent the set of digraphs that are either \(\theta\)-digraphs or Type \(j\) digraphs for \(j \in \{3, 4, 5\}\). These sets are referred to as the \(\infty\)-Family (resp. the \(\theta\)-Family). The following theorem, presented in \cite{flor2}, asserts that \(\mathcal{SCD}_3 = F_{\infty} \cup F_{\theta}\). This finding lays a foundational role in this paper, as it enables the exploration of these distinct families of digraphs to comprehensively understand the spectrum of digraphs with three complementarity eigenvalues.

\begin{teo}
\label{main_theorem}
Let $D$ be a digraph in $\mathcal{SCD}_3$. Then $D$ belongs either to the $\infty$-Family or to the $\theta$-Family, i.e., $\mathcal{SCD}_3 = F_{\infty} \cup F_{\theta}.$ Precisely $D$ is either an $\infty$-digraph, or $\theta$-digraph, or a Type $1$, or a Type $2$ , or a Type $3$, or a Type $4$, or Type $5$ digraph.

\end{teo}

Adapting the definition from \cite{Pinheiro2020}, we define a strongly connected digraph \(D\) as \textit{determined by its complementarity spectrum} (DCS) if any cospectral digraph \(H\) is either isomorphic to \(D\) or differs in the number of vertices from \(D\).\\

Based on the characterization results discussed earlier, it is straightforward to ascertain that isolated vertices and cycles qualify as DCS digraphs. However, \cite{flor} presents an example of non-isomorphic Type 4 digraphs sharing the same complementarity spectrum, illustrating that complementarity eigenvalues alone may not always uniquely determine a digraph. Consequently, this paper aims to investigate specific families of digraphs that exhibit the DCS property.\\

Consider a class of digraphs \(\mathcal{D}\), all of the same order, where each element \(D \in \mathcal{D}\) possesses a distinct complementarity spectrum. That is, for any \(H, D \in \mathcal{D}\), \(\Pi(H) = \Pi(D)\) implies \(H\) and \(D\) are isomorphic.\\ 

In this context, digraphs within \(\mathcal{D}\) are said to be \textit{determined by their complementarity spectrum in} \(\mathcal{D}\), abbreviated as \(\mathcal{D}\) is DCS. This designation suggests that while digraphs in these classes are DCS within their specific group, this attribute could pave the way to demonstrate their DCS status more broadly.\\

The remainder of this paper is dedicated to examining which of the seven digraph families outlined in Theorem \ref{main_theorem} are DCS. Section \ref{sec_Finfinito} employs basic and classical analysis tools to explore the characteristic polynomials of select digraphs, yielding insights into these families. Given the limitations of these tools for analyzing \(\theta\)-digraphs, Section \ref{sec_Ftheta} adopts alternative algebraic techniques to examine the roots of the characteristic polynomials.

\section{Classes of Digraphs in \(F_{\infty}\) Determined by Their Complementarity Spectrum}
\label{sec_Finfinito}

In this section, we explore the three families within \(F_{\infty}\): the \(\infty\)-digraphs, and Type 1 and Type 2 digraphs. Our analysis begins with a comparison of the spectral radii of two \(\infty\)-digraphs. Considering our interest in digraphs of identical order, adjusting the node count of one cycle necessitates a compensatory adjustment in the other cycle's size. This adjustment process delineates a hierarchy in the spectral radii of the digraphs, echoing a similar finding reported in \cite{Lin2012}.




\begin{teo}\label{teo:inftyDCS}
Let $r,s$ be integers such that $2\leq r\leq s$, $3\leq s$ and  $r+s-1=n$. Then $\rho(\infty(r,s))>\rho(\infty(r+1,s-1))$.

\end{teo}

\begin{proof}
Let $p_1(x)=x^n-x^{n-r}-x^{n-s}$ and $p_{2}(x)=x^n-x^{n-r-1}-x^{n-s+1}$ represent the characteristic polynomials of $\infty(r,s)$ and $\infty(r+1,s-1)$. Denote $\rho_1=\rho(\infty(r,s))$ and $\rho_2=\rho(\infty(r+1,s-1))$.\\

If $h(x)= p_{2}(x)-p_1(x)$, then
\[h(x)=-x^{n-r-1}+x^{n-r}-x^{n-s+1}+x^{n-s}=\]\[=x^{n-s}(-x^{s-r-1}+x^{s-r}-x+1)=x^{n-s}(x-1)(x^{s-r-1}-1).\]

It is easy to see that $s-r-1 \geq 1$, so $h(x)>0$ for all $x>1$. Given that $\rho_2>1$,  
\[0<h({\rho}_2)=p_{2}({\rho}_2)-p_1({\rho}_2)=-p_1({\rho}_2).\]

Therefore, it suffices to show that $p_1$ is strictly increasing in $(1, +\infty)$, because  $p_1({\rho}_2)<0=p_1(\rho_1)$ which implies $\rho_2 < \rho_1$.

Indeed, $p_1'(x)= nx^{n-1}-(n-r)x^{n-r-1}-(n-s)x^{n-s-1}$ which has exactly one positive root by Descartes' rule.
By virtue of the mean value theorem there exists $c$ in $(0,1)$ such that \[p_1'(c)=\frac{p_1(1)-p_1(0)}{1-0}=-1.\]

Additionally, $p_1'(1)=r-n+s=1>0$, and by Bolzano's theorem, there exists a root of $p_1'$ in $(c,1)$. Since $p_1'$ has only one positive root, $p_1'(x)>0$ for all $x$ in $(1,\infty)$. Then, $p_1$ is strictly increasing, completing the proof.
\end{proof}

\begin{cor}

Let $n\geq 3$ be an integer, let $\rho$ and ${\rho}'$ be the spectral radii of digraphs $\infty(r,s)$ and $\infty({r}',{s}')$, both with $n$ vertices. If $r<{r}'$ then $\rho > {\rho}'$.
\end{cor}

The results presented permit the establishment of a hierarchy based on the spectral radius within the class of $\infty$-digraphs of order $n$. Figure \ref{fig:comparacionInfinitosnCte} illustrates $\infty$-digraphs of order $9$, sequenced by their spectral radius, which coincidentally aligns with the lexicographic order.\\

\begin{figure}[h!]
\centering
\includegraphics[scale=0.075]{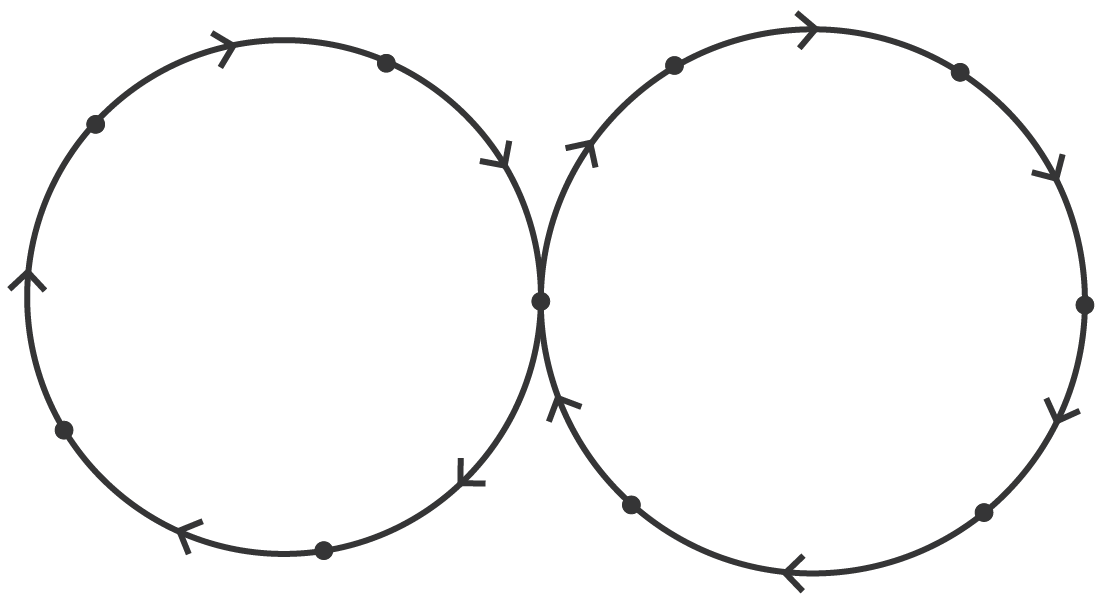}
\includegraphics[scale=0.075]{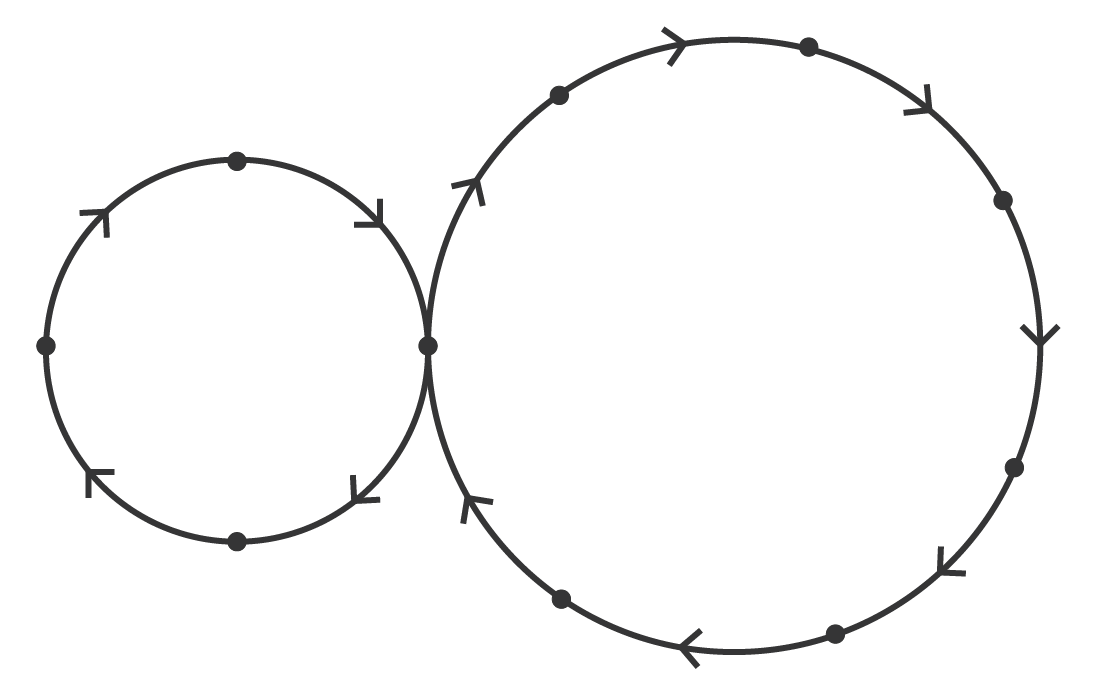}
\includegraphics[scale=0.075]{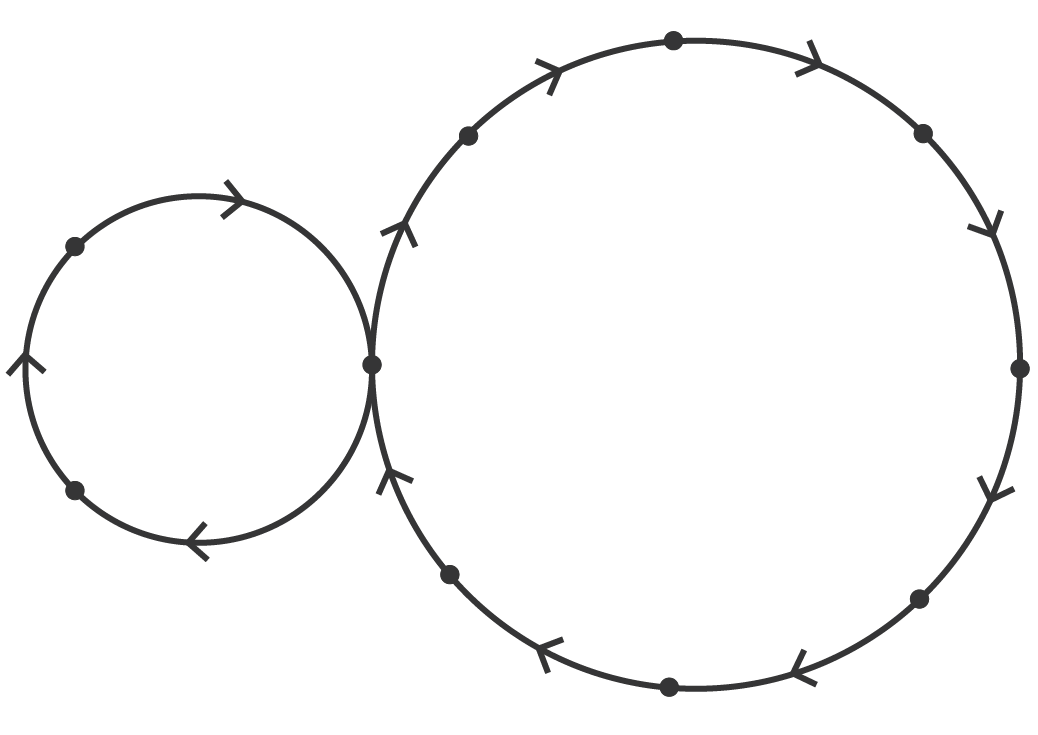}
\includegraphics[scale=0.075]{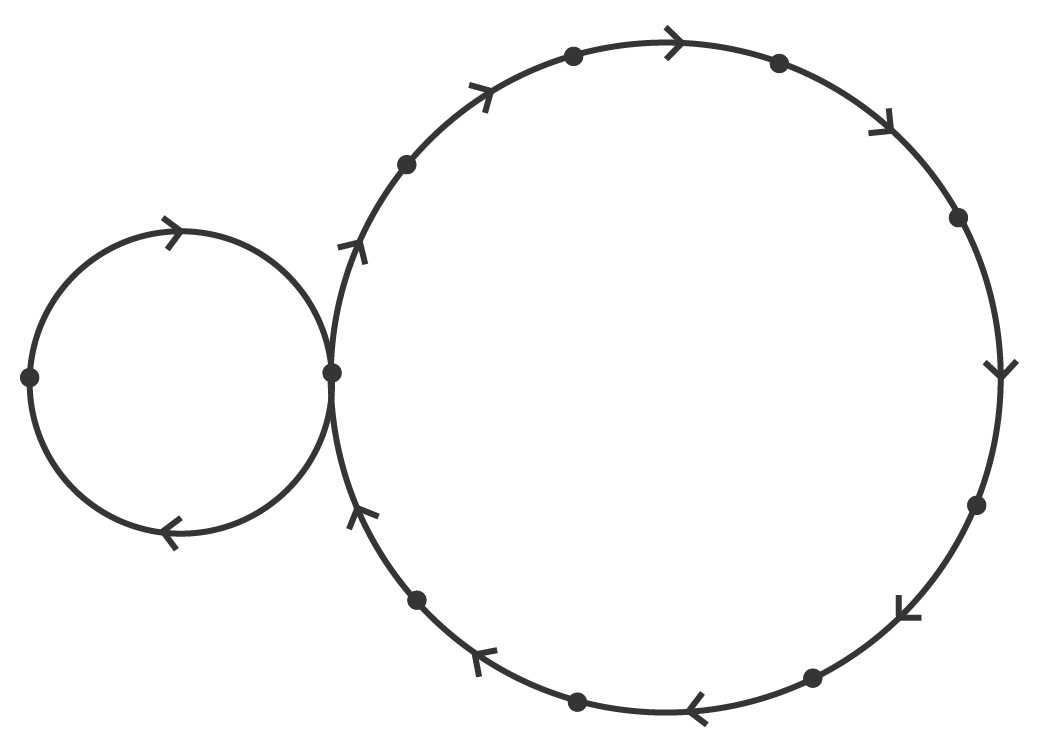}
\caption{$\rho(\infty(5,5))<\rho(\infty(4,6))<\rho( \infty(3,7))<\rho( \infty(2,8))$.}.
\label{fig:comparacionInfinitosnCte}
\end{figure}

As previously discussed, the analysis of the complementarity spectrum of these digraphs is simplified to examining their spectral radius. Consequently, we are able to establish the following result.

\begin{cor}
The class of $\infty$-digraphs is DCS.
\end{cor}

The following example shows that the class of Type 1 digraphs is not determined by their complementarity spectrum.

\begin{ej}
\end{ej}
\noindent

Let $2 \leq r < s$ be integers. The digraphs $D_{1a}(r,s)$ and $D_{1b}(r,s)$, as depicted in Figure~\ref{fig:tipo1ab}, are not isomorphic. Nonetheless, their characteristic polynomials are identical, with $p_{D_{1a}}(x) = x^n - x^{s-1} - x^{r-1} - 1 = p_{D_{1b}}(x)$. Therefore, both digraphs are cospectral and, specifically, complementarity cospectral, since they share the same spectral radius.\\


\begin{figure}[h!]
\centering
\includegraphics[scale=0.1]{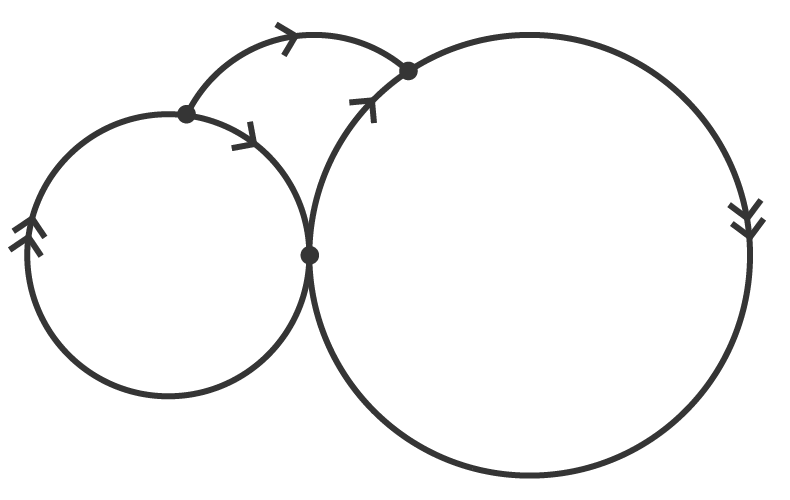}
\hspace{0.5cm}
\includegraphics[scale=0.15]{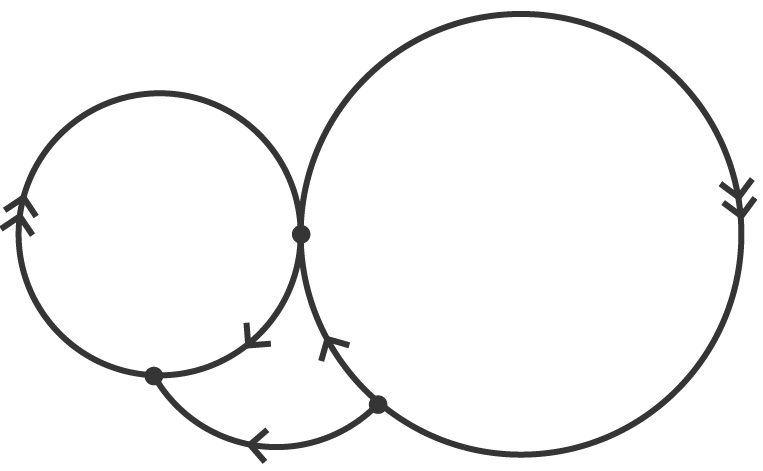}
\caption{{Digraphs $D_{1a}(r,s)$ and $D_{1b}(r,s)$}}
\label{fig:tipo1ab}
\end{figure}

The aforementioned example serves as an initial instance among several throughout this paper, illustrating that strongly connected digraphs are not generally distinguished by their complementarity spectra (DCS). This finding is documented in \cite{flor}.\\

Nonetheless, employing a reasoning similar to that presented in Theorem \ref{teo:inftyDCS}, it can be shown that the complementarity spectrum does effectively differentiate between the classes of Type $1a$ digraphs and Type $1b$ digraphs individually.

\begin{teo}
Let $r,s$ be integers such that $2\leq r\leq s$, $3\leq s $ and $r+s-1=n$. Then, $\rho(D_{1a}(r,s))>\rho(D_{1a}(r+1,s-1))$.
\end{teo}




\begin{cor}
Let $r,s,r',s'$ be integers such that $3\leq n$, $2\leq r\leq s$, $2\leq{r}'\leq{s}'$. Let $\rho$ and ${\rho}'$ be the spectral radii of digraphs $D_{1a}(r,s)$ and $D_{1a}({r}',{s}')$ both with $n$ vertices. If $r<{r}'$ then $\rho > {\rho'}$. 
\end{cor}

Similar to the approach for $\infty$-digraphs, Type $1a$ digraphs with $n$ vertices can be organized based on their spectral radius. Figure~\ref{fig:comparacion1a} demonstrates Type $1a$ digraphs of order $9$, arranged by their spectral radius, which, interestingly, corresponds to the lexicographic order.\\

\begin{figure}[h!]
\centering
\includegraphics[scale=0.075]{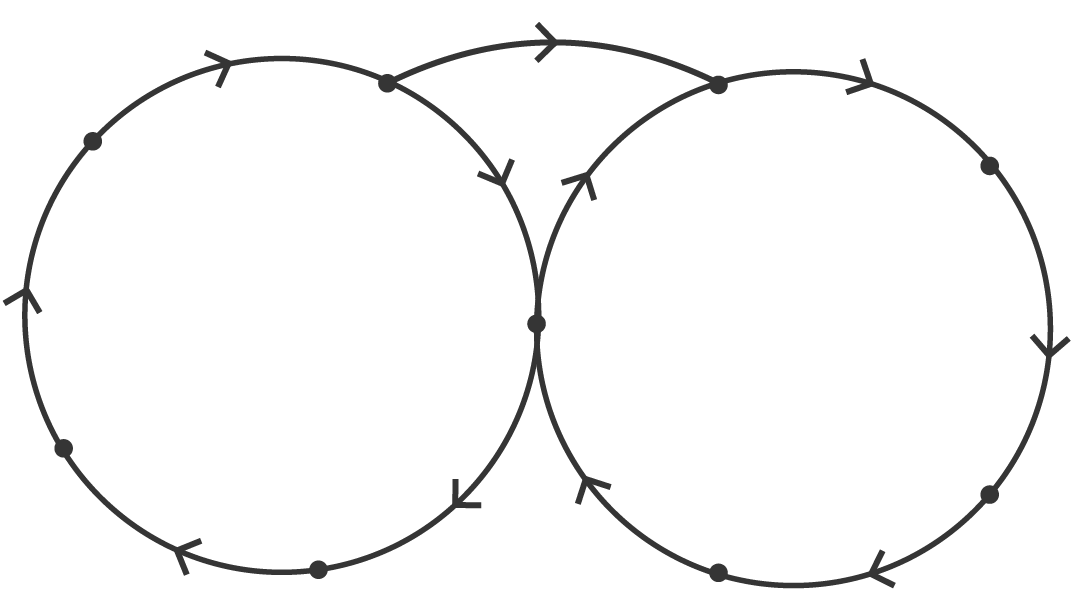}
\includegraphics[scale=0.075]{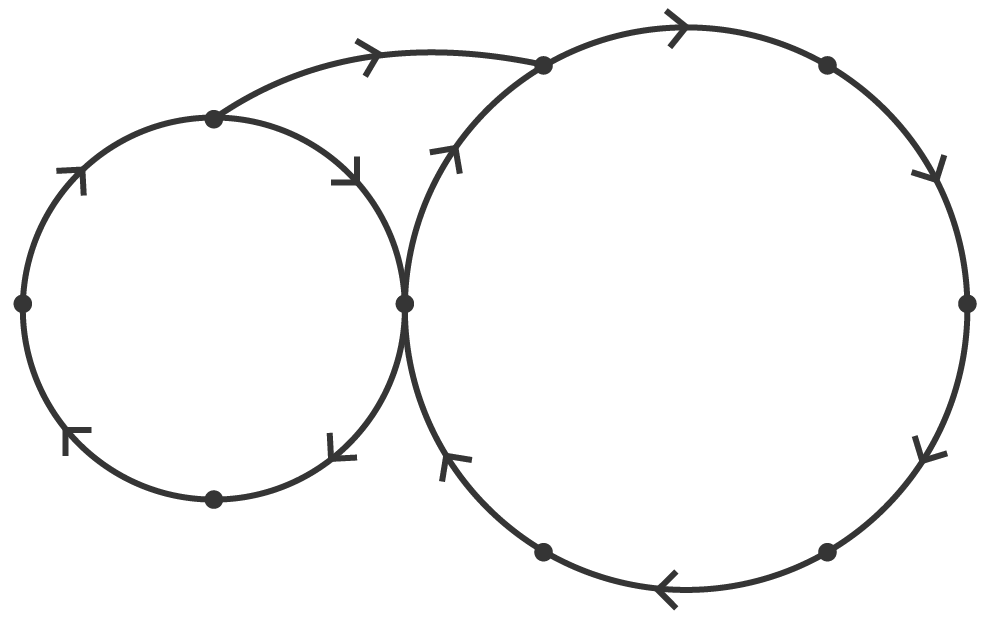}
\includegraphics[scale=0.075]{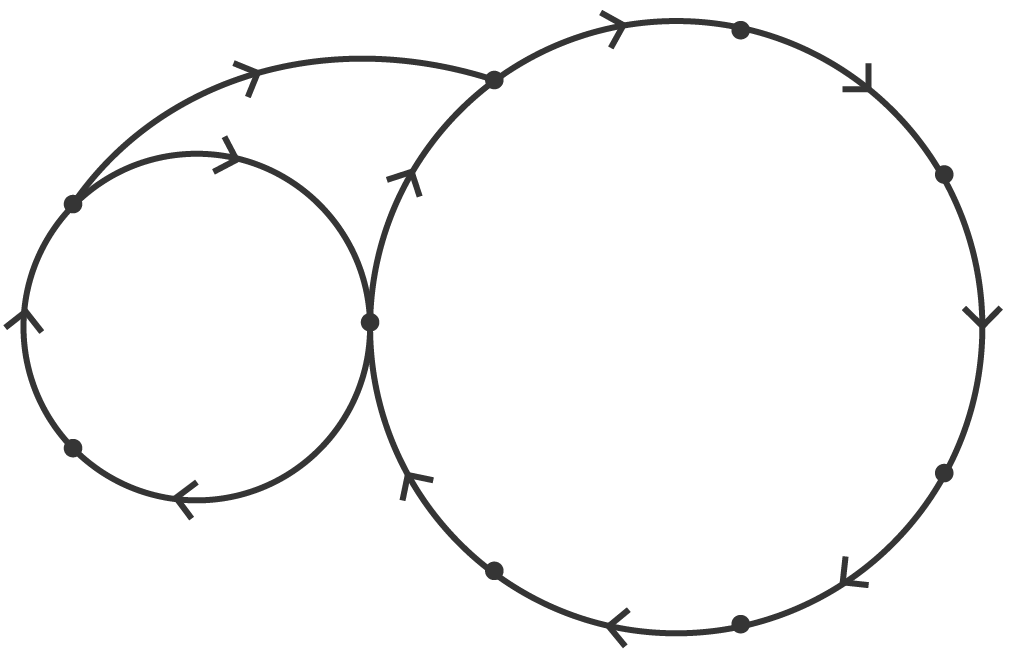}
\includegraphics[scale=0.075]{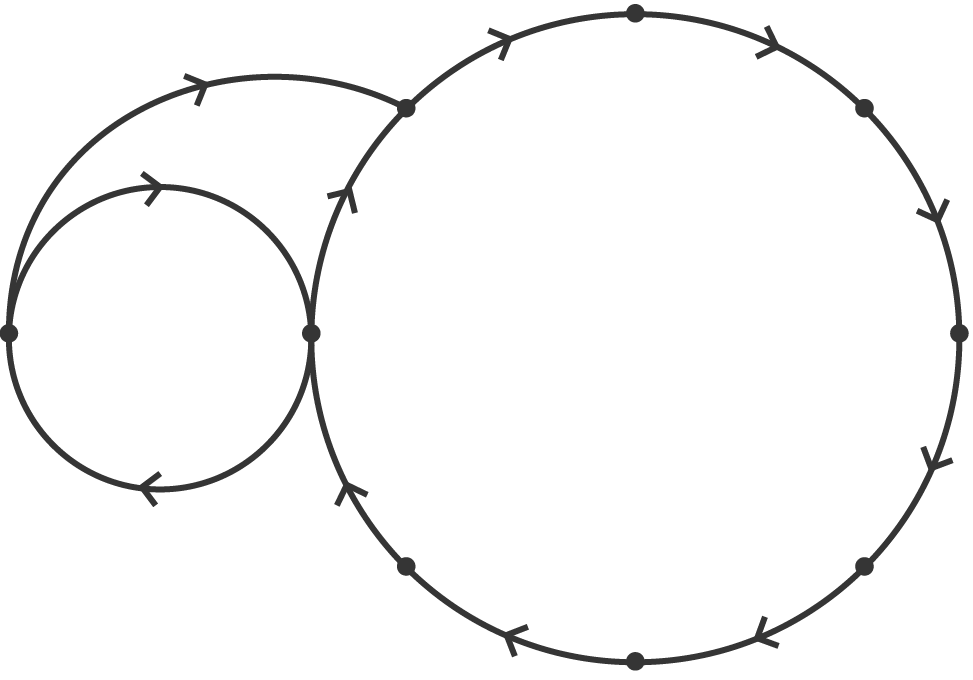}
\caption{$\rho(D_{1a}(5,5)) < \rho(D_{1a}(4,6)) < \rho(D_{1a}(3,7)) < \rho(D_{1a}(2,8))$.}.
\label{fig:comparacion1a}
\end{figure}

Once more, the analysis of the complementarity spectrum of these digraphs simplifies to examining their spectral radii, enabling us to assert the following conclusion.

\begin{cor}
The class of  Type $1a$ digraphs is DCS.
\end{cor}

The findings derived for Type $1a$ digraphs are similarly applicable to Type $1b$ digraphs through analogous proofs.\\

The examination of the family $F_\infty$ will be concluded with an analysis of Type $2$ digraphs.

\begin{teo}
Let $r,s$ be integers such that $2\leq r\leq s$, $3\leq s $ and $r+s-1=n$. Then, $\rho(D_{2}(r,s))>\rho(D_{2}(r+1,s-1))$.
\end{teo}

\begin{proof}


The proof follows the same steps as those outlined in Theorem \ref{teo:inftyDCS} until it is required to demonstrate that $p_1$ is strictly increasing over the interval $(1, +\infty)$. Specifically, the derivative of $p_1$, given by $p_1'(x) = nx^{n-1} - (n-r)x^{n-r-1} - (n-s)x^{n-s-1} - 1$, possesses exactly one positive root according to Descartes’ rule of signs. Given that $p_1'(1) = r - n + s - 1 = 0$, it follows that $p_1'(x) > 0$ for all $x$ in the interval $(1, \infty)$. Consequently, $p_1$ is strictly increasing in this domain, thereby completing the proof.

\end{proof}

\begin{cor}
Let  $3\leq n$, $2\leq r\leq s$, $2\leq{r}'\leq{s}'$ be integers, let $\rho$ and ${\rho}'$ spectral radii of $D_{2}(r,s)$ and $D_{2}({r}',{s}')$ both of order $n$. If $r<{r}'$ then $\rho > {\rho}'$.
\end{cor}

Similarly, Type $2$ digraphs of order $n$ can be organized based on their spectral radius, as illustrated in Figure~\ref{fig:comparacionTipo2}.


\begin{figure}[h!]
\centering
\includegraphics[scale=0.075]{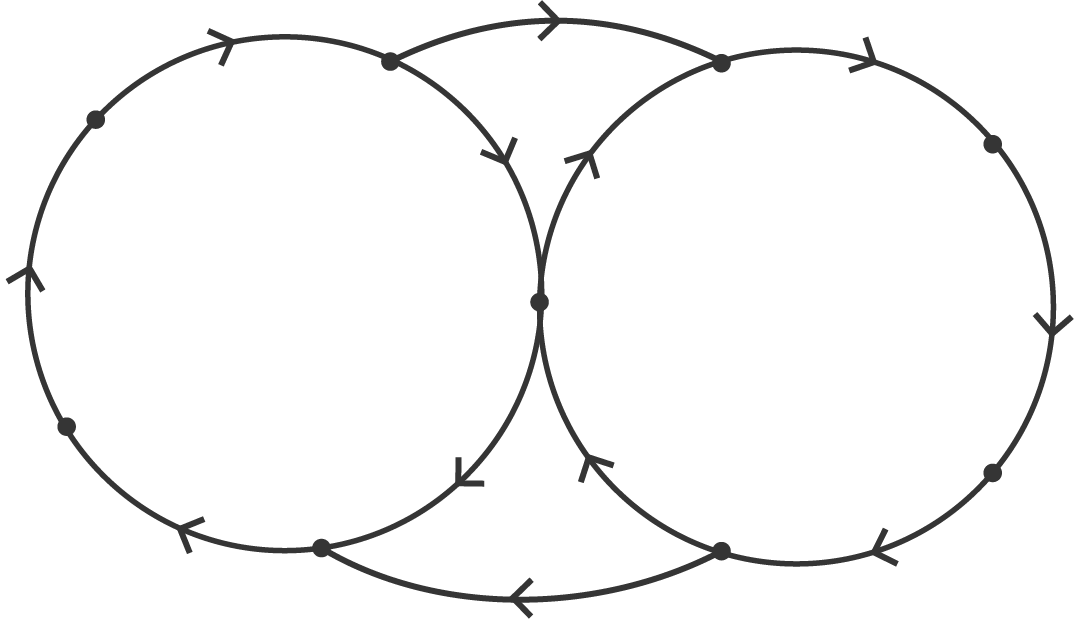}
\hspace{0.05cm}
\includegraphics[scale=0.075]{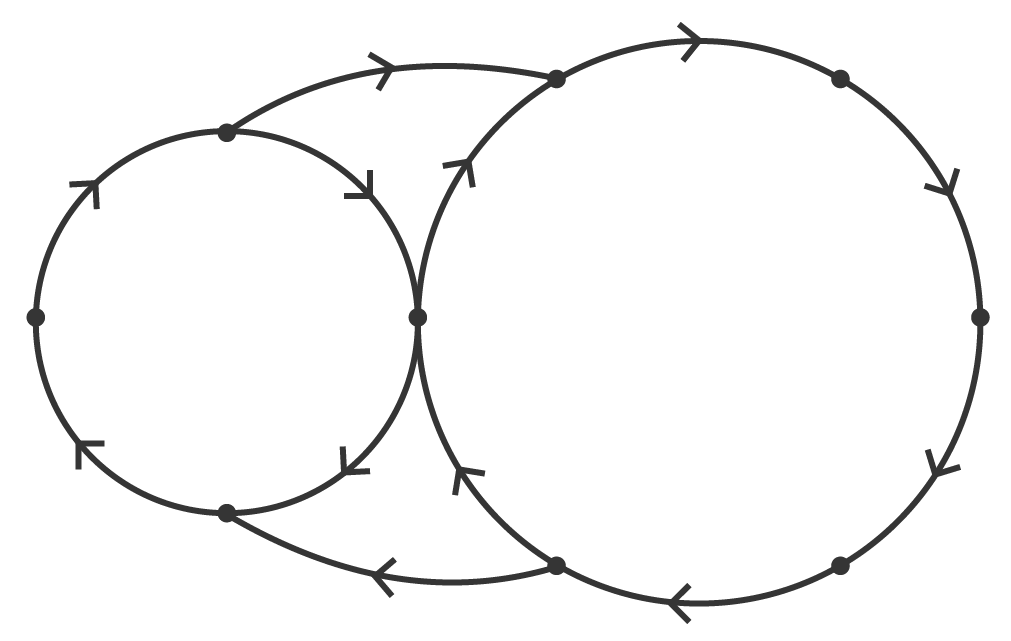}
\hspace{0.05cm}
\includegraphics[scale=0.075]{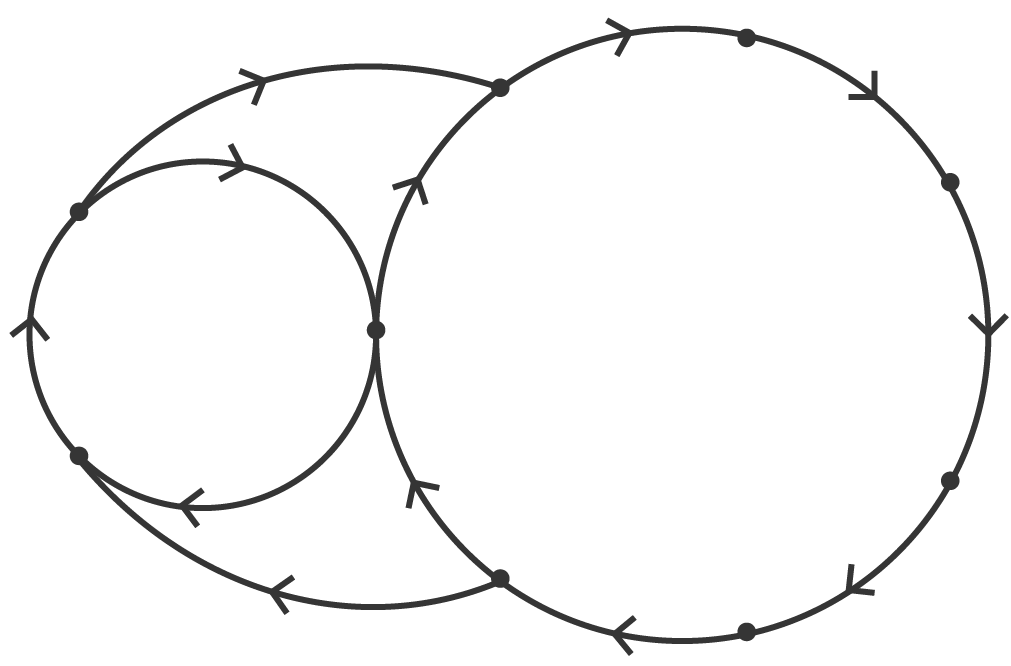}
\hspace{0.05cm}
\includegraphics[scale=0.075]{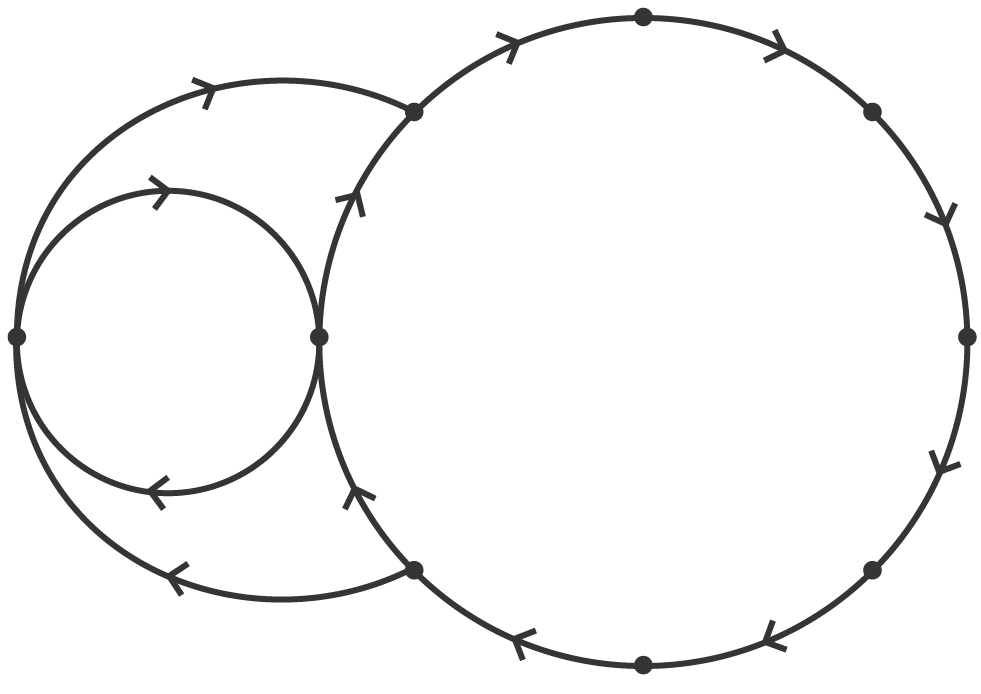}
\caption{$\rho(D_2(5,5))<\rho(D_2(4,6))<\rho(D_2(3,7))<\rho(D_2(2,8))$.}
\label{fig:comparacionTipo2}
\end{figure}

\begin{cor}
The class of Type $2$ digraphs is DCS.
\end{cor}

Therefore, the family \(F_{\infty}\) is categorized into four distinct sub-families that are distinguishable by their complementarity spectra (DCS). These sub-families include: \(\infty\)-digraphs, and the Type \(1a\), Type \(1b\), and Type \(2\) digraphs.

\section{Classes of digraphs in $F_\theta$ determined by their complementarity spectrum}\label{sec_Ftheta}



In this section, we establish that the class of $\theta$-digraphs is DCS by employing a strategy distinct from the one used in the previous section. To lay the groundwork, let us revisit some definitions and key results.\\ 

A number $\alpha \in \mathbb{C}$ is termed \textit{algebraic} if it is a root of a monic polynomial $f \in \mathbb{Q}[x]$, i.e., there exists a monic polynomial $f \in \mathbb{Q}[x]$ such that $f(\alpha) = 0$.\\

It is important to note that for any algebraic number $\alpha \in \mathbb{C}$, there exists a unique monic irreducible polynomial $\mathrm{Irr}_{\mathbb{Q}}(\alpha) \in \mathbb{Q}[x]$ that has $\alpha$ as a root.\\




Building upon the approach of the previous section, our objective is to select the characteristic polynomials of two distinct $\theta$-digraphs and demonstrate that their respective spectral radii differ. Note that the polynomial $P_\theta(x) = x^n - x^b - x^a$ can be factorized as $P_\theta(x) = x^a(x^{n-a} - x^{b-a} - 1)$. Given this factorization, the following polynomial, defined by specific values of $k$ and $m$, will be instrumental in our analysis.\\



\begin{definition}
Given two integers $k$ and $m$ such that $k > m > 0$, we define the polynomial $f_{k, m}(x)$ as $f_{k, m}(x) = x^k - x^m - 1$.
\end{definition}


\begin{definition}
Let the polynomial $h$ be defined by $h(x) = x^2 - x + 1$.
\end{definition}

{The following result, which is a particular case of Theorem 3 in \cite{ljun}, states conditions for $f_{n,m}$ to be irreducible, and will be useful in what follows.}

\begin{teo}\label{thm.paper}
Given the integers $k>m>0$, let $d = \gcd(k,m)$ and $k_1 = \frac{k}{d}$, $m_1 = \frac{m}{d}$. Then $f_{k,m}$  is irreducible except when $k_1+m_1\equiv 0\pmod{3}$ and $m_1\equiv 0\pmod{2}$ in which case we have the following decomposition as product of irreducible polynomials: $f_{k, m}(x) = h(x^d)\tilde{f}_{k, m}(x)$, where
$\tilde{f}_{k, m}$ is irreducible. Moreover $f_{k, m}(x) = f_{k_1, m_1}(x^d)$ {and $\tilde{f}_{k, m}(x) = \tilde{f}_{k_1, m_1}(x^d)$}.
\end{teo}

\begin{lema}\label{lema1}
Let $k>m>0$, and $k'>m'>0$ integers.  {If $f_{k,m}$ is reducible and 
$\tilde{f}_{k, m} = f_{k', m'}$
then  $k=5d$, $m=4d$, $k'=3d$ and $m'=d$ where  $d = \gcd(k,m)$.} 
\end{lema}
\begin{proof} Consider the polynomial
\begin{align*}
f_{k, m}(x) &= h(x^d)\tilde{f}_{k, m}(x) = h(x^d)f_{k', m'}(x) = (x^{2d}-x^d+1)(x^{k'}-x^{m'}-1)\\
&=x^{k'+2d}-x^{k'+d}-1 +(x^{m'+d}+x^{k'}+x^d-x^{m'+2d}-x^{m'}-x^{2d}).
\end{align*}
The term $x^{k' + 2d}$ emerges as the highest degree term in the expanded polynomial, implying $k' + 2d = k$.\\ 

Given that $f_{k,m}$ has exactly three monomials and $-x^{k'+d}$ cannot be written off with $x^{m'+d},x^{k'}$ nor $x^d$ because $k'+d > m'+d, k', d$. Then $k'+d=m$ and $x^{m'+d}+x^{k'}+x^d-x^{m'+2d}-x^{m'}-x^{2d}=0$. Also,  $m'+d \neq m'+2d, m'$ allowing to conclude that $m'+d = 2d$. Finally $m'=d$, $k = 5d$ and $m = 4d$.
\end{proof}

\begin{lema}\label{lema2}
If $\gcd(k, m) = 1$ and $f_{k, m}$ is reducible then 
{$\tilde{f}_{k, m}(x) \equiv -x-1 \pmod{(x^2)}$.}
\end{lema}
\begin{proof}
By Theorem \ref{thm.paper} if $f_{k,m}$ is reducible then $m\geq 2$ because $m$ is even. Also by Theorem  \ref{thm.paper}, $f_{k,m}(x) = h(x)\tilde{f}_{k,m}(x)$ because $\gcd(k, m) = 1$, then $f_{k,m}(x)
\equiv -1\pmod{(x^2)}$, and $f_{k, m}(x) \equiv (-x+1)(ax+b)\pmod{(x^2)}\equiv (a-b)x+b\pmod{(x^2)}$.
Hence $a = b = -1$.
\end{proof}\\

The previous lemma can be generalized for any $k,m$ due to the identity $\tilde{f}_{k, m}(x) =\tilde{f}_{k_1,m_1}(x^d)$. 
\begin{obs}
    \label{obs:Lema2Generalizado} If $f_{k, m}$ is reducible then $\tilde{f}_{k, m}(x) 
\equiv -x^d-1  \pmod{(x^{2d})}$.
\end{obs}

\begin{lema}\label{lema3}
Given the integers $k>m>0$, $k'>m'>0$, such that $f_{k, m}$, $f_{k', m'}$ are reducible. If $\tilde{f}_{k, m} = \tilde{f}_{k', m'}$ then $k = k'$, $m = m'$.
\end{lema}
\begin{proof} Let $d = \gcd(k, m)$ and $d' = \gcd(k', m')$. By the Observation \ref{obs:Lema2Generalizado}, $-x^d$ is the monomial with the least positive exponent in $\tilde{f}_{k, m}$, similarly for $-x^{d'}$ and $\tilde{f}_{k',m'}$, then $d = d'$ and $k = k'$. Moreover $f_{k,m}(x) = (x^{2d} - x^d + 1)\tilde{f}_{k, m}(x) = (x^{2d} - x^d + 1)\tilde{f}_{k, m'}(x) = f_{k,m'}(x)$ so
 $m = m'$.
\end{proof}

\begin{lema}\label{lema4}
Let $k>0$ then the polynomial $x^k - 2$ is irreducible.
\end{lema}
\begin{proof}
Apply Eisenstein criterion for the prime number $2$.
\end{proof}

\begin{lema}\label{lema5}
Let $n, a, b\in\N$ with $n>b\geq a\geq 0$, $n\geq a+b+2$. Let $f(x) = x^n - x^b - x^a$, and $\rho$ its only real positive root. If $a = b$, then $\Irr_\rho(x) = x^{n-a} - 2$. {If $b>a$ and  $f_{n-a, b-a}$ is irreducible then $\Irr_\rho = f_{n-a, b-a}$. Otherwise,
$\Irr_\rho = \tilde{f}_{n-a, b-a}$.
}
\end{lema}


\begin{proof}
Consider two cases based on the relationship between $a$ and $b$:

\textbf{Case 1:} $a = b$. In this case, the polynomial simplifies to $f(x) = x^a(x^{n-a}-2)$. Given that $\rho > 0$ is a root, it must satisfy $x^{n-a}-2=0$, which is an irreducible polynomial by Lemma \ref{lema4}.

\textbf{Case 2:} $b > a$. Here, $f(x)$ becomes $x^a(x^{n-a}-x^{b-a}-1) = x^a f_{n-a,b-a}(x)$. It is clear that $\rho$ is a root of $f_{n-a, b-a}(x)$. If $f_{n-a, b-a}(x)$ is irreducible, then $\Irr_\rho = f_{n-a, b-a}$. If not, we express $f_{n-a, b-a}(x)$ as $h(x^d)\tilde{f}_{n-a,b-a}(x)$, where $d = \gcd(n-a, b-a)$. Since $h(x)$ lacks real roots, $h(x^d)$ also has no real roots, indicating that $\rho$ must be a root of $\tilde{f}_{n-a,b-a}(x)$.
\end{proof}\\

We are now in condition to prove the main result of this section.

\begin{teo}
\label{teo:rama}
Let $n, a_1, a_2, b_1, b_2\in\N$, such that $(a_1, b_1) \neq (a_2, b_2)$ and $n> b_i\geq a_i\geq0$, $n\geq a_i+b_i+2$ with $i=1, 2$. 
If $\rho_1,\rho_2$ are the positive real roots of $f_i(x) = x^n-x^{b_i}-x^{a_i}$, then $\rho_1\neq\rho_2$.
\end{teo}
\begin{proof}
It suffices to see that  $\Irr_{\rho_1}(x) \neq \Irr_{\rho_2}(x)$. It is clear that $f_i(x) = x^{a_i}(x^{n-a_i} - x^{b_i-a_i} - 1)$. We have three possible cases:
\begin{enumerate}
\item $a_1 = b_1$ and $a_2 = b_2$:\\
Suppose that $\Irr_{\rho_1} = \Irr_{\rho_2}$. Then, $\Irr_{\rho_i}(x) = x^{n-a_i}-2$ by Lemma \ref{lema5}. Hence, $x^{n-a_1} - 2 = x^{n-a_2} - 2$, implying $a_1 = a_2$, which contradicts $(a_1, b_1) \neq (a_2, b_2)$.

\item $a_1 = b_1$, $b_2 > a_2$:\\
By Lemma \ref{lema5} we have $\Irr_{\rho_1}(x) = x^{n-a_1}-2$ and $\Irr_{\rho_2} = f_{n-a_2,b_2-a_2}$ or $\Irr_{\rho_2} = \tilde{f}_{n-a_2, b_2-a_2}$.
The independent term of $\Irr_{\rho_1}$, $\Irr_{\rho_2}$ is $2$, $-1$ respectively. Hence $\Irr_{\rho_1}\neq \Irr_{\rho_2}$.

\item $b_1 > a_1$, $b_2 > a_2$:\\
Suppose that $\Irr_{\rho_1} = \Irr_{\rho_2}$.
There are three cases to consider:
\begin{enumerate}
\item[I.] If $\Irr_{\rho_1} = f_{n-a_1, b_1-a_1}$ and $\Irr_{\rho_2} = f_{n-a_2, b_2-a_2}$ then $a_1 = a_2$ and $b_1 = b_2$ {which is a contradiction}.

\item[II.] If $\Irr_{\rho_1} = \tilde{f}_{n-a_1, b_1-a_1}$ and $\Irr_{\rho_2} = \tilde{f}_{n-a_2, b_2-a_2}$, then by Lemma \ref{lema3}, $n-a_1 = n-a_2$ and $b_1 - a_1 = b_2 - a_2$ hence $b_1 = b_2$ and $a_1 = a_2$ {which is a contradiction}.

\item[III.] If $\Irr_{\rho_1} = \tilde{f}_{n-a_1, b_1-a_1}$ and $\Irr_{\rho_2} = f_{n-a_2, b_2-a_2}$, then by Lemma \ref{lema1}, $n-a_1 = 5d$, $b_1 - a_1 = 4d$, $n - a_2 = 3d$ and $b_2 - a_2 = d$.
Consequently we deduce that $n = 5d+a_1 = 3d+a_2$, leading to $a_2 = 2d+a_1$. Also, $b_2 = d + a_2 = 3d + a_1$. The condition $n \geq b_2+a_2+2$ simplifies to $5d+a_1 \geq 3d+a_1 + 2d + a_1 +2= 5d+2a_1 + 2$
implying $0\geq a_1+2$, which is a contradiction.
\end{enumerate}
\end{enumerate}
\end{proof}\\

{The following result is a direct consequence of Theorem \ref{teo:rama}.}\\
\begin{teo} {The class of $\theta$-digraphs is DCS.}
\end{teo}
\begin{proof}
Let $n>a_i\geq b_i\geq0$ be integers for $i=1,2$. Consider $\theta _i=\theta(a_i, b_i, n-a_i-b_i+2)$ and $\rho_i=\rho(\theta_i)$ with $i=1, 2$.\\

$\Pi(\theta_1)=\Pi(\theta_2)$ if, and only if, $\rho_1=\rho_2$. Applying Theorem \ref{teo:rama} implies

\[\rho_1=\rho_2  \Leftrightarrow  (a_1,b_1)=(a_2,b_2) \Leftrightarrow \theta_1 \cong \theta_2.\]
\end{proof}\\

This marks the culmination of our examination of \(\theta\)-digraphs.\\

In what follows, we present a family of examples that shows that the class of Type $3$ digraphs is not DCS. To start with, we will prove the following useful lemma.


\begin{lema}
Let $D_3(i,j)$ and $D_3(i',j')$ be Type $3$ digraphs with $n$ vertices. Then, $D_3(i,j)\cong D_3(i',j')$ if, and only if, $(i,j) = (i',j')$.
\end{lema}

\begin{proof}
First, observe that the only vertices $x$ in $D_3(i,j)$ verifying
\[
gr^{-}(x)=1 \quad gr^{+}(x)=2,
\]
are $1$ and $i-1$. Moreover, we can distinguish one from each other given the simple path that connects them, composed exclusively by vertices satisfying $gr^{-}(x)=gr^{+}(x)=1$ (see Figure~\ref{fig:contraejemploTipo3}). Then, $D_3(i,j)\cong D_3(i',j')$ implies that $i=i'$. Furthermore, $p_{D_3(i,j)}(x)=p_{{D}_3(i,j')}(x)$ implies $x^n-x^{i-2}-x^{j-i}-1=x^n-x^{i-2}-x^{j'-i}-1$, hence $j=j'$.

\begin{figure}[h]
\centering
\includegraphics[scale=0.20]{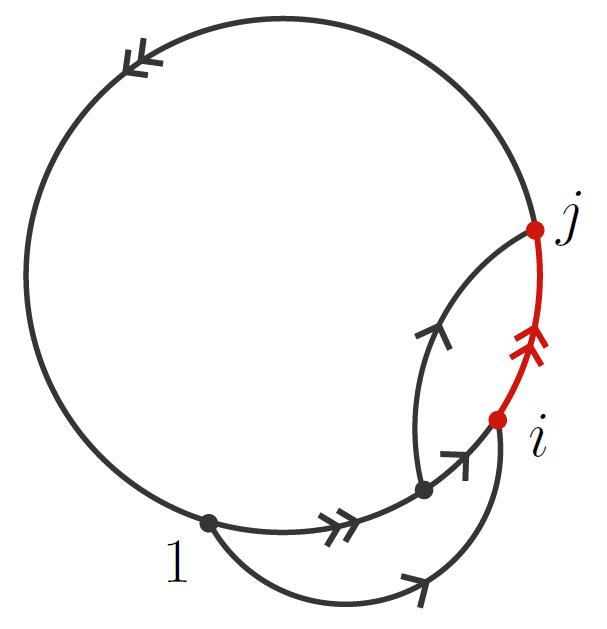}
\caption{Digraph $D_3(i,j)$ and the simple path mentioned colored in red.}
\label{fig:contraejemploTipo3}
\end{figure}
\end{proof}

\begin{ej}
\end{ej}
\noindent

Let \(n \geq 5\) be an integer. Consider the digraphs \(D_3(3,n)\) and \(D_3(n-1,n)\), which are not isomorphic as established by the preceding lemma. Nonetheless, it is observed that their characteristic polynomials coincide:
\[
p_{D_3(3,n)}(x) = x^n - x^{n-3} - x - 1 = p_{D_3(n-1,n)}(x).
\]
Therefore, despite their structural differences, \(D_3(3,n)\) and \(D_3(n-1,n)\) are cospectral, and in particular, complementarity cospectral.\\


Similar to our exploration of Type 3 digraphs, we will introduce an example pertaining to Type 4 digraphs that cannot be distinguished solely by their complementarity spectra. But before, let us prove the following lemma.\\

\begin{lema}
Let $D_4(i)$ be a Type $4$ digraph of order $n$ defined as $\vec{C}_n\cup\{(2,1),(i+1,i)\}$ with $i=3,\hdots, n-1$. Then, $D_4(i)\cong D_4(j)$ if, and only if, $i=j$ or $i+j=n+2$.
\end{lema}

\begin{proof}
The digraph $D_4(i)$, represented in Figure~\ref{fig:contraejemploTipo4}, has two distinguished simple paths $\vec{P}_{i-1}$ and $\vec{P}_{n-i+1}$. If $i\neq j$ then,
\[D_4(i)\cong D_4(j) \Leftrightarrow
\begin{cases}
\text{ the lengths of the simple paths $\vec{P}_{i-1}$ and $\vec{P}_{n-j+1}$ coincides,}\\
\text{ the lengths of the simple paths $\vec{P}_{j-1}$ and $\vec{P}_{n-i+1}$ coincides.}
\end{cases}
\]
Consequently, both conditions can be reduced to $i+j=n+2$. Then, for $j\neq i, n+2-i$ it is deduced that $D_4(i)\ncong D_4(j)$, but \[p_{D_4(i)}(x)=x^n-2x^{n-2}+x^{n-4}-1=p_{D_4(j)}(x).\]

\begin{figure}[h!]
\centering
\includegraphics[scale=0.16]{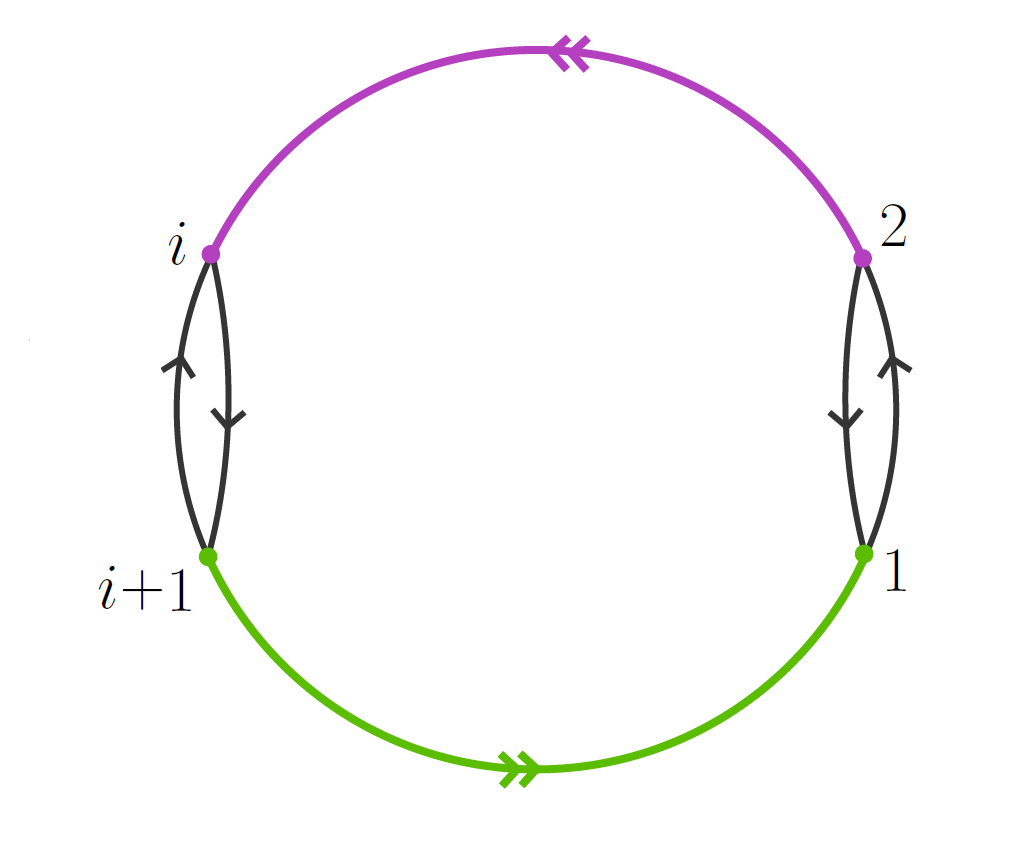}
\caption{Digraph $D_4(i)$ and the simple paths $\vec{P}_{i-1}$ colored in violet and $\vec{P}_{n-i+1}$ colored in green.}
\label{fig:contraejemploTipo4}
\end{figure}

\end{proof}

The following example is documented in \cite{flor}.\\

\begin{ej}
\end{ej}


\noindent Consider the digraphs \(D_4(3)\) and \(D_4(j)\), where \(j \neq 3, n-1\), non-isomorphic as established in the preceding lemma. Observe that their characteristic polynomials coincide:
\[
p_{D_4(3)}(x) = x^n - 2x^{n-2} + x^{n-4} - 1 = p_{D_4(j)}(x).
\]
Consequently, \(D_4(3)\) and \(D_4(j)\) are cospectral, notably complementarity cospectral.\\

Before analyzing the complementarity spectrum of Type $5$ digraphs, let us observe that they are not uniquely determined by the parameters $i$ and $j$, but by the sizes of the cycles in it: $4\leq r\leq s\leq t<n$ (see Figure~\ref{fig:ciclosEnTipo5}). In terms of these parameters, we have that 
\[p_{D_5}(x)=x^n-x^{n-r}-x^{n-s}-x^{n-t}-2.\]

\begin{figure}[h!]
\centering
\includegraphics[scale=0.15]{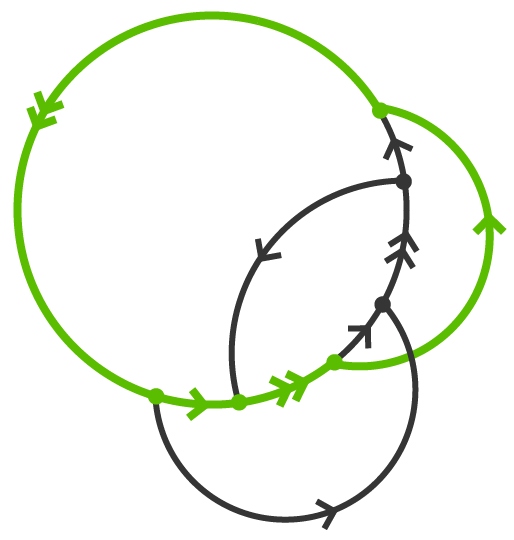}
\includegraphics[scale=0.15]{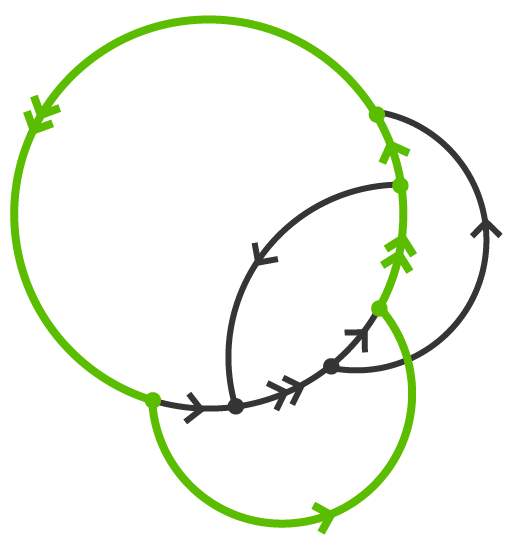}
\includegraphics[scale=0.15]{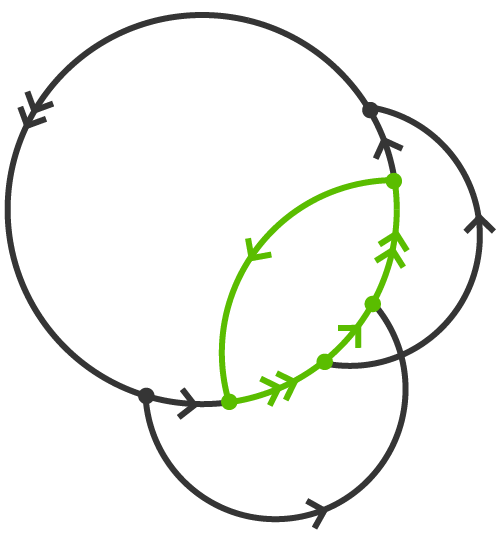}
\caption{Cycles $\vec{C_r}$, $\vec{C_s}$, $\vec{C_t}$ in a Type $5$ digraph colored in green.}
\label{fig:ciclosEnTipo5}
\end{figure}

\begin{prop}
\label{propTipo5}
Let $D$ and $D'$ be Type $5$ digraphs. $D$ and $D'$ are isomorphic if, and only if, they are cospectral.
\end{prop}

\begin{proof}
If $D$ and $D'$ are isomorphic, they are cospectral. We will prove the reciprocal: let $r\leq s\leq t<n$ and $r'\leq s'\leq t'<n$ be the sizes of the cycles in $D$ and $D'$ respectively. Consider the characteristic polynomials equality,

 \[p_D(x)=x^n-x^{n-r}-x^{n-s}-x^{n-t}-2=x^n-x^{n-r'}-x^{n-s'}-x^{n-t'}-2=p_{D'}(x),\]
 
then, $r=r'$, $s=s'$ and $t=t'$. Let us see that those identities imply the isomorphism of both digraphs.\\

We will see that the size of the cycles ($r,s$ and $t$) gives place to a unique Type $5$ digraph constructively. To start with, observe that each vertex of the Type $5$ digraph belongs to exactly both of these cycles, thus \[r+s+t=2n.\]

Next, consider the sub-digraph $H$ with the arcs of the cycles $\vec{C_r}$ and $\vec{C_s}$ (see Figure~\ref{fig:H}).\\

\begin{figure}[h!]
\centering
\includegraphics[scale=0.15]{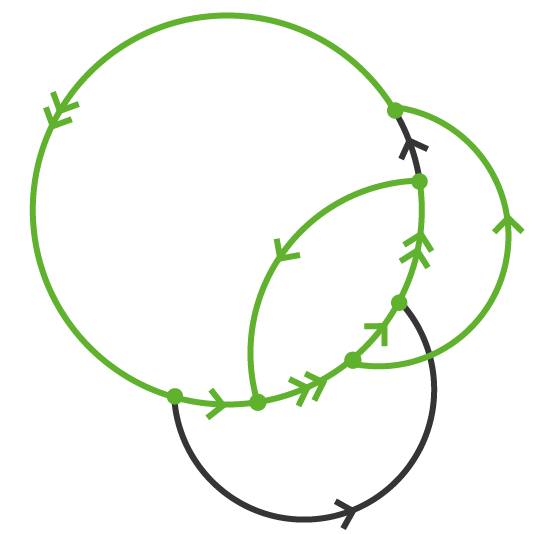}
\hspace{.5cm}
\includegraphics[scale=0.15]{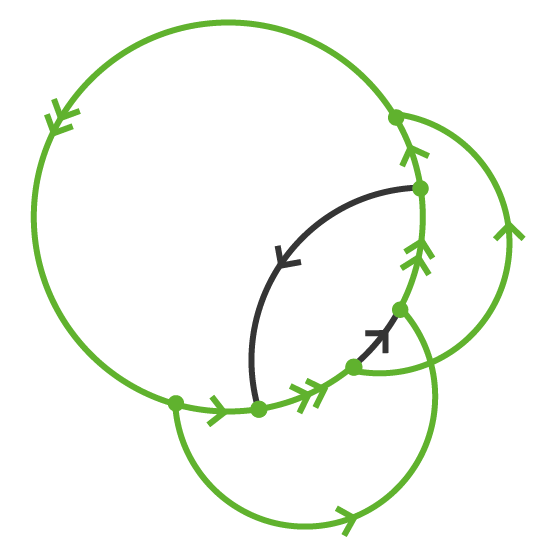}
\hspace{.5cm}
\includegraphics[scale=0.15]{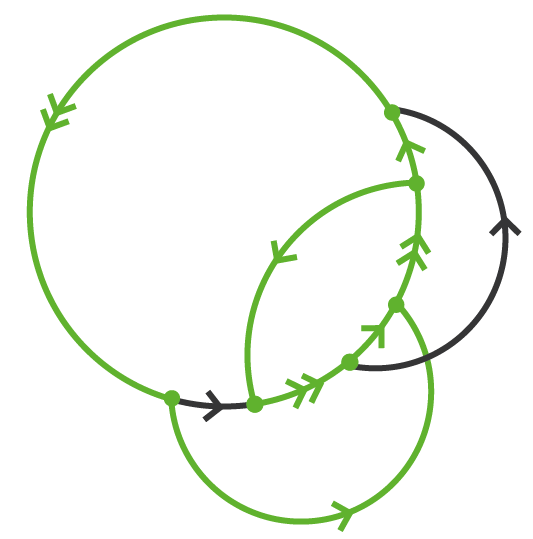}
\caption{Possible digraphs $H$ pointed out in green.}
\label{fig:H}
\end{figure}

We assume our analysis falls within the scope of the first case (left on Figure~\ref{fig:H}), denoted as $H$, with the understanding that other cases follow a similar rationale.\\

As illustrated in Figure~\ref{fig:HconColores}, we identify $H$ as $\theta(r+t-n, s+t-n, n-t-2)$. According to Case 1 of Theorem 4.2 in \cite{flor2}, the only arcs that may be appended to $H$ within the framework of the $\theta$-digraph are $(v_a,w_b)$ and $(v_b,w_a)$.

\begin{figure}[h!]
\centering
\includegraphics[scale=0.15]{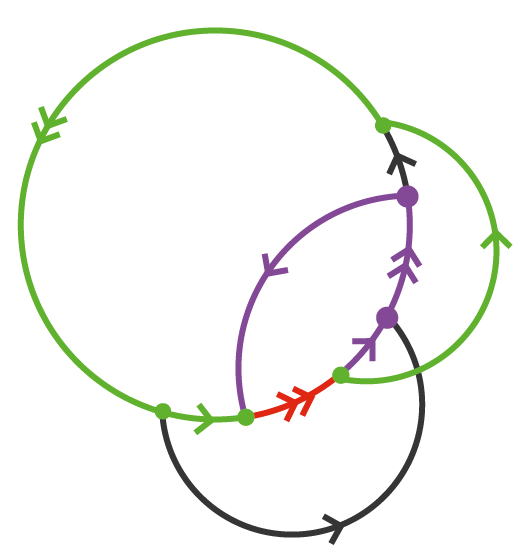}
\caption{Digraph $H=\theta(a,b,c)$ where $\vec{P}_{a+2}$, $\vec{P}_{b+2}$ and $\vec{P}_{c+2}$ are respectively colored in red, violet and green.}
\label{fig:HconColores}
\end{figure}

\begin{figure}[h!]
\centering\includegraphics[scale=0.15]{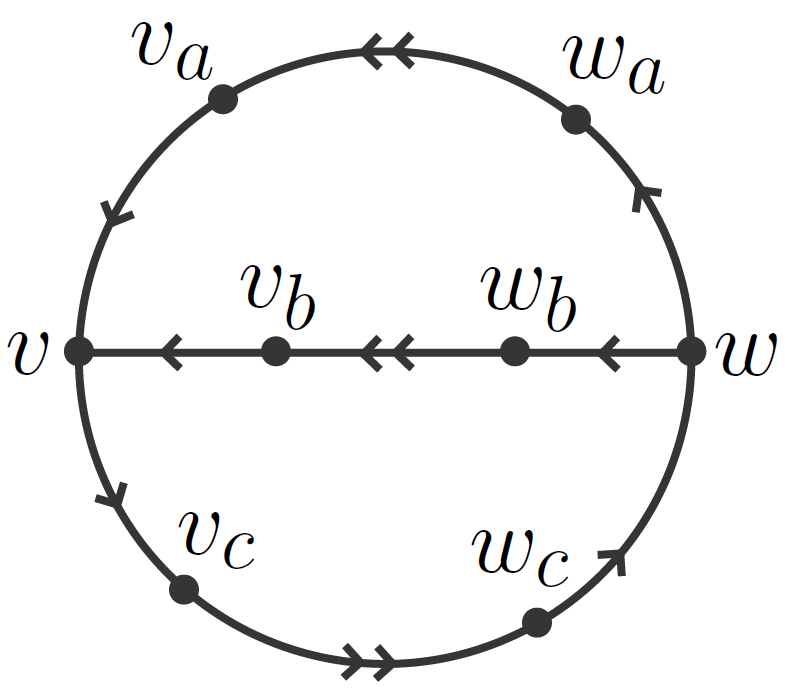}
\caption{Digraph $\theta(a,b,c)$ and some distinguished vertices.}
\label{fig:thetaConLetras}
\end{figure}


Furthermore, observe that $H$ contains $n+1$ arcs. Given that a Type $5$ digraph comprises $n+3$ arcs, the addition of both specified arcs is necessary, resulting in a unique Type $5$ configuration.
\end{proof}

\newpage
The following corollaries are direct consequences of the previous demonstration.

\begin{cor}
Let $D$ be a Type $5$ digraph and $r \leq s \leq t<n$ the sizes of the cycles in it.
Then, \[r+s+t=2n.\]
\end{cor}

\begin{cor}
Let $D$ and $D'$ be a Type $5$ digraph. Let $r \leq s \leq t<n$ and $r'\leq 's \leq t' <n$ be the sizes of the cycles in $D$ and $D'$ respectively.\\

$D$ and $D'$ are isomorphic if, and only if, $r=r'$, $s=s'$ and $t=t'$.
\end{cor}

\begin{ej}
\end{ej}
\noindent Digraphs $D_5(4,10)$ and $D_5(8,10)$ with $10$ vertices are cospectral and by Proposition \ref{propTipo5} we have that $D_5(4,10) \cong D_5(8,10)$.\\

In what follows, we will use the parameters $r\leq s\leq t<n$ to identify a Type $5$ digraph transitioning away from the previous notation $i,j$.

\begin{teo}\label{partial_tipo5}
Let $n \geq 3$ be an integer. Let $\rho_1$ and $\rho_2$ be the spectral radii of $D_5(r,s,t)$ and $D_5(r',s',t')$. If one of the following holds:
\begin{itemize}
    \item $r\leq r'\leq s'\leq t'\leq s \leq t$, 
    \item $r\leq  s \leq r'\leq s' \leq t' \leq t$. 
\end{itemize}

then $\rho(D_5(r,s,t)) < \rho(D_5(r',s',t'))$.
\end{teo}

\begin{proof}
Let $p_1(x)=x^n-x^{n-r}-x^{n-s}-x^{n-t}-2$ and $p_{2}(x)=x^n-x^{n-r'}-x^{n-s'}-x^{n-t'}-2$ be
the characteristic polynomials of $D_5(r,s,t)$ and $D_5(r',s',t')$.
Denote $\rho_1=\rho(D_5(r,s,t))$ and $\rho_2=\rho(D_5(r',s',t'))$.\\ 

If $h(x)= p_{1}(x)-p_2(x)$, then
\[h(x)=-x^{n-r}-x^{n-s}-x^{n-t}+x^{n-r'}+x^{n-s'}+x^{n-t'}.\]
Let us see that $h(x)< 0$ for all $x>1$. The Descartes' rule of signs implies that $h(x)$ has no more than two positive roots (counted with multiplicity). Moreover, $h(1)=0$ and $h'(1)=-(n-r)-(n-s)-(n-t)+(n-r')+(n-s')+(n-t')=0$ implying that $1$ is a double root of $h$.\\

Since $\rho_2>1$, it follows that
\[0> h(\rho_2)=p_{1}(\rho_2)-p_2(\rho_2)=p_1(\rho_2).\]
Hence, it will be enough to prove that $p_1$ is strictly increasing in the interval $(1, +\infty)$, since $p_1(\rho_2) <0 = p_1(\rho_1)$.\\

The derivative of $p_1(x)$ is the polynomial 

$$p_1'(x)= nx^{n-1}-(n-r)x^{n-r-1}-(n-s)x^{n-s-1}-(n-t)x^{n-t-1}$$ 

which has exactly one
positive root by Descartes’ rule. This root should be equal to $1$ given that $p_1'(1)=n-(n-r)-(n-s)-(n-t)=r+s+t-2n=0$. Then, $p_1'(x)>0$ for all $x$ in the interval $(1,+\infty)$. Thus, $p_1$ is strictly increasing, completing the proof.
\end{proof}\\

Building upon the insights gained from the previous demonstration, we are now in a position to establish the following result:

\begin{cor}
\label{cor:Tipo5rCte}
Let $n\geq 3$ be an integer. Let $\rho=\rho(D_5(r,s,t))$ and ${\rho}'=\rho(D_5(r',s',t'))$, both with $n$ vertices. If $r=r'$ and $s< s'$ then $\rho < {\rho}'$.
\end{cor}

\begin{cor}
\label{cor:ordenLexTipo5rCte}
The class of digraphs \[\mathcal{D}(r_0)=\{D_5(r_0,s,t): s,t \text{ non negatives}\}.\]
is DCS.
\end{cor}

Analogous results could be established for \[\mathcal{D}(s_0)=\{D_5(r,s_0,t): r,t \text{ non negatives}\} \]
\[ \mathcal{D}(t_0)=\{D_5(r,s,t_0): r,s \text{ non negatives}\}.\]

However, there are Type $5$ digraphs whose spectral radii are not comparable using the previous Lemma, for example $D_5(6,10,12)$ and $D_5(8,9,11)$, both with $14$ vertices.\\


\begin{teo}
\label{conj:rama}
Let $n\in\N$, $5\leq n\leq 200$, $4\leq r_k \leq s_k \leq t_k < n$, such that $r_k + s_k + t_k = 2n$, for $k = 1, 2$ and $(r_1, s_1, t_1) \neq (r_2, s_2, t_2)$. 
If $\rho_k$ is the only real positive root of $x^n - x^{n - r_k} - x^{n - s_k} - x^{n - t_k} - 2$, for $k = 1, 2$, then $\rho_1 \neq \rho_2$.
\end{teo}
\begin{proof}


To verify this claim, we utilized the computational approach delineated in \cite{rama_code}, executed within the Sage mathematical software environment \cite{sagemath}.

Our verification strategy mirrors the approach adopted in Theorem~\ref{teo:rama}. For a specified value of $n$ within the theorem's hypotheses, we enumerate all feasible combinations of $r, s, t$ that also satisfy the theorem's conditions. Subsequently, we define the corresponding polynomial for each combination and determine the irreducible polynomial of its positive root. Upon comparing these polynomials, we observed no instances of repetition.
\end{proof}\\

{Based on this last result, on the partial ordering given in Theorem \ref{partial_tipo5}, and other experimental evidence, we conjecture that this is true for all $n\geq 5$.}

\begin{conjecture}
Let $n\in\N$, $5\leq n$, $4\leq r_k \leq s_k \leq t_k < n$, such that $r_k + s_k + t_k = 2n$, for $k = 1, 2$ and $(r_1, s_1, t_1) \neq (r_2, s_2, t_2)$. 
If $\rho_k$ is the only real positive root of $x^n - x^{n - r_k} - x^{n - s_k} - x^{n - t_k} - 2$, for $k = 1, 2$, then $\rho_1 \neq \rho_2$.
\end{conjecture}

{Observe that the validity of this conjecture implies that the class of Type 5 digraphs is DCS.}

\section{Conclusions}\label{sec_conclu}

In this paper, we have studied the families of digraphs with three complementarity eigenvalues and described which of these seven families are determined by their complementarity spectrum (DCS). This is, that every two non-isomorphic digraphs in a DCS family have a different complementarity spectrum.\\

We proved that the families: $\infty$, $\theta$, and Type 2 digraphs are DCS.\\

Also, since the family of Type 1 digraphs has a natural symmetry, we have given a partition of this Type 1 family into two other DCS sub-families.\\

Likewise, examples were provided establishing that Type 3 and Type 4 families are not DCS.\\

Furthermore, partial results were given for the Type 5 family, which seem to indicate that it is also DCS. This last statement is posed as a conjecture at the end of the paper.\\

Since we focused on digraphs with three complementarity eigenvalues, the study of the spectral radius was crucial in our results. This allowed to connect our work with other papers outside the complementarity spectrum literature. Indeed, the study of spectral radii of bicyclic digraphs was studied in \cite{Lin2012}, with partial results for some specific families. The present manuscript extends those results, proving, for example, that the $\theta$-digraphs, which are bicyclic, can in fact be distinguished by the spectral radii.\\

The theorems on the spectral radii of other families presented in this paper may be the initial results towards the study of spectral radii for digraphs with more than two cycles.

\bibliographystyle{elsarticle-num}
\bibliography{biblio}

\begin{thebibliography}{10}
\expandafter\ifx\csname url\endcsname\relax
  \def\url#1{\texttt{#1}}\fi
\expandafter\ifx\csname urlprefix\endcsname\relax\def\urlprefix{URL }\fi
\expandafter\ifx\csname href\endcsname\relax
  \def\href#1#2{#2} \def\path#1{#1}\fi

\bibitem{flor2}
D.~Bravo, F.~Cubría, M.~Fiori, V.~Trevisan, Characterization of digraphs with
  three complementarity eigenvalues, Journal of Algebraic Combinatorics 57
  (2023) 1173--1193.

\bibitem{flor}
D.~Bravo, F.~Cubría, M.~Fiori, V.~Trevisan, Complementarity spectrum of
  digraphs, Linear Algebra and its Applications 627 (2021) 24--40.

\bibitem{Collatz1957}
L.~Von~Collatz, U.~Sinogowitz, Spektren endlicher grafen, Abhandlungen aus dem
  Mathematischen Seminar der Universit{\"a}t Hamburg 21~(1) (1957) 63--77.

\bibitem{Harary1971}
F.~Harary, C.~King, A.~Mowshowitz, R.~C. Read, Cospectral graphs and digraphs,
  Bulletin of the London Mathematical Society (1971) 321--328.

\bibitem{Pinheiro2020}
L.~K. Pinheiro, B.~S. Souza, V.~Trevisan, Determining graphs by the
  complementary spectrum, Discussiones Mathematicae Graph Theory 40~(2) (2020)
  607--620.

\bibitem{Lin2012}
H.~Lin, J.~Shu, A note on the spectral characterization of strongly connected
  bicyclic digraphs, Linear Algebra and its Applications 436~(7) (2012)
  2524--2530.

\bibitem{Baillon2020}
J.-B. Baillon, A.~Seeger, New results on pareto spectra, Linear Algebra and its
  Applications 588 (2020) 338--363.

\bibitem{Seeger99}
A.~Seeger, Eigenvalue analysis of equilibrium processes defined by linear
  complementarity conditions, Linear Algebra and its Applications 292~(1-3)
  (1999) 1--14.

\bibitem{Adly2015}
S.~Adly, H.~Rammal, A new method for solving second-order cone eigenvalue
  complementarity problems, Journal of Optimization Theory and Applications
  165~(2) (2015) 563--585.

\bibitem{Facchinei2007}
F.~Facchinei, J.-S. Pang, Finite-dimensional variational inequalities and
  complementarity problems, Springer Science \& Business Media, New York, 2007.

\bibitem{Pinto2008}
A.~Pinto~da Costa, A.~Seeger, Cone-constrained eigenvalue problems: theory and
  algorithms, Computational Optimization and Applications 45 (2010) 25--57.

\bibitem{Pinto2004}
A.~Pinto~da Costa, J.~Martins, I.~Figueiredo, J.~J{\'u}dice, The directional
  instability problem in systems with frictional contacts, Computer Methods in
  Applied Mechanics and Engineering 193~(3-5) (2004) 357--384.

\bibitem{ljun}
W.~Ljunggren, On the irreducibility of certain trinomials and quadrinomials,
  Math. Scand. 8 (1960) 65--70.
\newblock \href {https://doi.org/10.7146/math.scand.a-10593}
  {\path{doi:10.7146/math.scand.a-10593}}.

\bibitem{rama_code}
G.~Rama, Companion code for the article families of digraphs determined by the
  complementarity spectrum,
  \newline\url{https://gitlab.fing.edu.uy/grama/spectral} (2023).

\bibitem{sagemath}
{The Sage Developers}, {S}ageMath, the {S}age {M}athematics {S}oftware {S}ystem
  ({V}ersion 9.3), {\tt https://www.sagemath.org} (2021).

\end{thebibliography}





\end{document}